\documentclass[a4paper,11pt]{article}

\usepackage{amsmath,amssymb,amsthm,color}
\usepackage[dvips]{graphicx}
\usepackage[latin1]{inputenc}

\newtheorem{theorem}{Theorem}[section]
\newtheorem{prop}[theorem]{Proposition}
\newtheorem{lemma}[theorem]{Lemma}

\newtheorem{defi}[theorem]{Definition}
\newtheorem{rema}[theorem]{Remark}

\newcommand{\C}{\ensuremath{\mathbb C}}

\newcommand{\N}{\ensuremath{\mathbb N}}

\newcommand{\dd}{\text{d}}

\newcommand{\DD}{\mathcal{D}}

\newcommand{\FF}{\mathcal{F}}

\newcommand{\MM}{\mathcal{M}}
\newcommand{\UU}{\mathcal{U}}

\newcommand{\fraction}[2]{\displaystyle\frac{#1}{#2}}

\begin{document}

\title{Normal forms of foliations and curves defined by a function\\ with a generic tangent cone.}
\author{Yohann {\sc Genzmer}, Emmanuel {\sc Paul}}
%\thanks{\textbf{Keywords}: holomorphic foliation, moduli of curve, singularities.\\
%\indent\textbf{A.M.S. class.}: 34M35, 32S65, 32G13}

\maketitle

\bigskip

\abstract{\scriptsize We first describe the local and global moduli spaces of germs of foliations defined by analytic functions in two variables with $p$ transverse smooth branches, and with integral multiplicities (in the univalued holomorphic case) or complex multiplicities (in the multivalued ''Darboux'' case). We specify normal forms in each class. Then we study on these moduli space the distribution $\mathcal C$ induced by the following equivalence relation: two points are equivalent if and only if the corresponding foliations have the same analytic invariant curves up to analytical conjugacy. Therefore, the space of leaves of $\mathcal C$ is the moduli space of curves. We prove that $\mathcal C$ is rationally integrable. These rational integrals give a complete system of invariants for these generic plane curves, which extend the well-known cross-ratios between branches.}
\footnote{\textbf{Keywords}: holomorphic foliation, moduli of curve, singularities.\\
\indent\textbf{A.M.S. class.}: 34M35, 32S65, 32G13}

%\tableofcontents

\section*{Introduction}

%\textcolor{red}{Le fait que deux $f_1...f_p$ sont topologiquement conjugués est dans Paul, classification %topologiques des formes logarithmiques }*

We consider a germ of holomorphic function at the origin of $\C^2$ whose irreducible decomposition is given by
$$f_1^{n_1}\cdots f_p^{n_p}.$$
We shall make use of the condensed notation $f^{(n)}$ where $(n):=(n_1,\cdots n_p)$. For instance, $f^{(1)}$ is the reduced function related to $f$. This data defines three different mathematical objects: a function with values in $(\C,0)$, an holomorphic foliation (we are not interested in values but only in the partition of a neighborhood of 0 in $\C^2$ by the fibers of $f$), and an analytic curve defined by the equation $f=0$ (we are not yet concerned with other fibers). The corresponding analytic equivalence relations are the three following ones:
\begin{eqnarray*}
	f_0\sim_f f_1&\Leftrightarrow& \exists \phi \in \mbox{Diff}\ (\mathbb{C}^2,0),\ f_1=f_0\circ \phi.\\
	f_0\sim f_1&\Leftrightarrow& \exists \phi \in \mbox{Diff}\ (\mathbb{C}^2,0),\ \psi \in \mbox{Diff}\ (\mathbb{C},0),\ \psi(f_1)=f_0\circ \phi.\\
	f_0\sim_c f_1&\Leftrightarrow& \exists \phi \in \mbox{Diff}\ (\mathbb{C}^2,0),\ \exists u \in \mathcal{O}_2, u(0)\neq 0, uf_1=f_0\circ \phi.
\end{eqnarray*}
We are here mainly interested in the second one, and therefore denote it by the simplest notation $\sim$. One can define similar classifications (topological, formal...) changing the class of the conjugacy $\phi$. We fix here the following topological class: we suppose that the branches $f_i=0$ are non singular, with distinct tangencies. Therefore, one can desingularize the foliation defined by $f$ by only one blowing-up. We denote by $\mathcal{T}^{(n)}$ the set of functions which satisfy this hypothesis and admits $(n)$ as multiplicities. All the foliations defined by a function in $\mathcal{T}^{(n)}$ are topologically equivalent (see \cite{Paul1}).

The first goal of this paper is to give a complete description of the moduli space $\mathcal{T}^{(n)}/\sim$, by normal forms. We shall extend this description to the class of foliations defined by a Darboux function: $f_1^{\lambda_1}\cdots f_p^{\lambda_p}$, where the multiplicities $\lambda_i$ are complex numbers. The second one - the description of $\mathcal{T}^{(n)}/\sim_c$ or \emph{the Zariski problem in the generic case} -, is approached here in the following way: we want to describe on the previous moduli space the distribution $\mathcal C$ whose leaves correspond to the foliations such that the related invariant analytic sets --the separatrix of the foliation-- define the same curve up to $\sim_c$. Therefore the moduli space for curves is the quotient space of this distribution, and our goal is to describe it. We first compute its generic dimension, and recover in a different way a result of J.M. Granger (\cite{Granger}). Then we describe the whole moduli space of curves by proving that this distribution is integrable by rational first integrals, which define a complete system of invariants.
We hope that this strategy will be efficient in more degenerated cases. Indeed, some results --as the method to compute infinitesimal generators of the distribution $\mathcal C$: theorem (\ref{qh_engendre})-- admit a natural extension in any other topological classes.

\bigskip
 
{\bf Statements of our results.}
Let $A$ be the set of $(p-3)\times (p-3)$ upper-triangular matrices $a=(a_{k,l})$ such that the entries of its first line are distincts, and different from $0,\ 1$.
We consider the following family of ''triangular'' functions:
 $$N_a^{(n)}=x^{n_1}y^{n_2}(y+x)^{n_3}\prod_{l=1}^{p-3}(y+\sum_{k=1}^{l}a_{k,l}x^k)^{n_{l+3}},\ a\in A $$
where the $a_{k,l}$ are the coefficients of the matrix $a$.
%, and $\{n_i\}$ a family of non-vanishing integers.
We first recall a prenormalization of curves (theorem (\ref{prenorm_curves}), see also \cite{Granger}):

\bigskip
\textbf{Theorem.} \textit{For each element $f$ in $\mathcal{T}^{(n)}$, there exists an element $a$ in $A$ such that $f\sim_c N_a^{(1)}$.}
\bigskip

\noindent Note that the matrix $a$ given by this theorem is not unique. Nevertheless, the entries of its first line (up to some permutations of branches) are defined in a unique way: they correspond to the cross-ratios of the tangent cone, normalized by the choice of the three first branches.

\noindent In order to study the moduli space of foliations $\mathcal{T}^{(n)}/\sim$, we first prove, using tools developed by JF Mattei (in \cite{Mattei1}), a ''local'' result (i.e. for families):

\bigskip
\textbf{Theorem.} 
\textit{The family of functions
 $\left\{N^{(n)}_a\right\}_{a\in(A,a^0)}$
defines a semi-universal equireducible unfolding of the foliation $\FF_0$ defined by $N^{(n)}_{a^0}$. }
\bigskip

\noindent In order to deal with the global moduli space of foliations, we construct a path $f_t$, $t$ in $[0,1]$, in $\mathcal{T}^{(n)}$, which connect a function $f_1$ defining a given foliation $\mathcal{F}_1$ to the homogeneous function defining the foliation $\mathcal{F}_0$ with same tangent cone, and such that for any $t\neq 0$, the foliation $\mathcal{F}_t$ is analytically equivalent to $\mathcal{F}_1$. Applying the previous local result to $\mathcal{F}_0$, we obtain prenormal forms for foliations :
%(\ref{prenorm_foliations}):

\bigskip
\textbf{Theorem.}  \textit{For each element $f$ in $\mathcal{T}^{(n)}$, there exists an element $a$ in $A$ such that $f\sim N_a^{(n)}$.}
\bigskip

\noindent It now suffices to detect which prenormal forms give rise to equivalent foliations. If we number the branches and we require that the classification keep invariant this numbering - we shall mention it as \emph{marked moduli space} - , we can prove that the foliations defined by $N_a$ and $N_b$ are equivalent if and only if $a$ is equivalent to $b$ under the following action 
of $\mathbb{C}^*$ on the lines of $A$:
$$\lambda\cdot (a_{k,l})=(\lambda^{k-1}a_{k,l}).$$
Actually, the normal forms $N_a^{(n)}$ satisfy a functional relation with respect to this action, which is going to be fundamental for our purpose
$$N_a^{(n)}\left(\lambda x,\lambda y\right)=\lambda^{|n|}N_{\lambda\cdot a}(x,y)$$

\noindent We denote by $\mathbb{P}A$ the quotient of $A^{*}$ by this action, where $A^{*}$ is the subset of matrices such that the lines of indices $\geq 2$ does not vanish everywhere. We can summarize the previous results by the following theorem:

\bigskip
\textbf{Theorem.}   \textit{The marked moduli space of foliations defined by $f$ in $\mathcal{T}^{(n)}$, punctured by the class of the homogeneous foliation, is the weighted projective space $\mathbb{P}A$. Its dimension is $\frac{(p-2)(p-3)}{2}-1$, where $p$ is the number of branches of $f$.} 
\bigskip

\noindent If we allow permutations of the branches, that is to say if we remove the marking, we have furthermore to identify the matrices obtained by permutation of the columns which preserve the triangular form, i.e. permutations of the columns inside blocks of columns of the same lenght. As we shall see through a simple example, this ''free'' moduli space might be really hard to describe.

\bigskip

\noindent In the last section, we give an exhaustive description of the partition $\mathcal{C}$ defined on $A$ by the following equivalence relation: $a$ and $b$ in $A$ are equivalent if and only if $N_{a}^{(n)}\sim_c N_{b}^{(n)}.$ 

\bigskip
\textbf{Theorem.}   \textit{There exists on $A$ a foliation $\mathcal{C}$ with the following property: $a$ and $b$ are in the same leaf of $\mathcal{C}$ if and only if $N_{a}^{(n)}\sim_c N_{b}^{(n)}$. This foliation is completely integrable by rational first integrals. Its generic codimension $\tau$ is $\frac{(p-2)^2}{4}$ if $p$ is even, or $\frac{(p-1)(p-3)}{4}$ if $p$ is odd.} 
\bigskip  

\noindent The formula of the dimension $\tau$ in the above statement was already known from a work of Granger \cite{Granger} but our method is completely different and is expected to be appliable in a more general context. 

\noindent This theorem is proved thanks to a description of an involutive family of vector fields generating the distribution $\mathcal{C}$. It appears that, using a finite determinacy property, this family can even be explicitely computed with the help of a computer. We also present an algorithm which determines a complete family of first integrals and give an example with nine irreducible branches, and without details, the case of ten branches.

\section{Prenormal form for the curve.}

The aim of this section is to prove the following result:

\begin{theorem}\label{prenorm_curves} Let $S$ be a germ of curve with $p$ irreducible smooth transversal components. Then, up to a change of coordinates, the curve $S$ is given by 
 $$N_a^{(1)}=xy(y+x)\prod_{l=1}^{p-3}(y+\sum_{k=1}^{l}a_{k,l}x^k)=0, $$
for a $a$ in $A$.
\end{theorem}
\noindent This theorem has already been proved by J.M. Granger in \cite{Granger}. These normal forms are not unique, and the previous author gives more precise normal forms to obtain unicity. Since the proof of J.M. Granger is really \emph{algebraic}, following Kodaira's approach of the moduli problem, we give a geometric proof of this theorem, in the flavour of the remainder of the present paper.

\bigskip
%\subsection{Finite determinacy of a germ of curve.}
Let $S$ be a germ of analytical curve in $(\mathbb{C}^2,0)$. Following a classical definition, \emph{the tangent cone} of $S$ is the set 
$$\mathcal{C}_1(S):=\tilde{S}\cap \DD$$
where $\tilde{S}$ is the strict transform of $S$ by the standard blowing-up of the origin $E_1:(\MM,\DD)\rightarrow (\mathbb{C}^2,0)$, $\DD$ refering to the exceptional divisor $E_1^{-1}(0)$ (see \cite{Sei} or \cite{MaMo}). By induction, we denote by $E_h$ the map $E_{h-1}\circ E^{h}$, where $E^h$ stands for the blowing-up centered at $\mathcal{C}_{h-1}(S)$. Moreover, we set $\mathcal{C}_{h}(S):=\widetilde{E_h^{-1}(S)}\cap E_{h}^{-1}(0)$. In what follows, the \emph{complete cone} of $S$, denoted by $\mathcal{C}(S)$, is defined to be 
$$ \mathcal{C}(S)=\lim_{\leftarrow} \bigcup_{i=1}^h \mathcal{C}_h(S).$$ 
The set $\mathcal{C}_h(S)$ is called the \emph{component of height $h$} of $\mathcal{C}(S)$. Note that the component of height $1$ coincides with the classical tangent cone. 

\begin{prop}[Finite determinacy]\label{detfin} Let $S$ be any germ of analytical curve. There exists an integer $N(S)$ depending only on the topological class of $S$ such that the following property holds: for any curve $S'$ topologically equivalent to $S$, if the components of the complete cones of $S$ and $S'$ of height less than $N(S)$ coincide, then $S$ and $S'$ are analytically equivalent.
\end{prop}

\begin{proof}
Since $S$ and $S'$ are topologically equivalent, their desingularization processes have same dual tree \cite{Zariski1}. Let us choose $N(S)$ at least bigger than the number of blowing-up in the desingularization process of $S$.  On account of the requirement on the cones, the desingularizations of $S$ and $S'$ are now actually equal. Denote by $E$ the common desingularization map. Let $f$ and $f'$ in $\mathcal{O}_0$ be some respective reduced equations of $S$ and $S'$ at the origin of $\mathbb{C}^2$. In the neighborhood of any point $c$ of the exceptionnal divisor $\DD$ belonging also to the complete cone of the curves, the strict transforms of $S$ and $S'$ are given by the functions
$$\tilde{f}_c:=h_c^{-p}E^*f \textup{  and  } \tilde{f}'_c:=h_c^{-p}E^*f'$$
where $h_c$ stands for a local equation of $\DD$ and $p$ for the common multiplicity of $E^*f$ and $E^*f'$ along the divisor at the point $c$ (this multiplicity is a topological invariant according to \cite{Zariski2}). From the assumption on the cones, we have the following lemma

\begin{lemma}
For $N$ big enough, there exist an integer $m$, a neighborhood $V_c$ of $c$ and a germ of unity $u_c$ such that 
$$\tilde{f}_c-u_c\tilde{f}'_c\in \mathcal{O}(-m\DD)(V_c)$$ and $m$ tends to infinity with $N$.  
\end{lemma}

\begin{proof}
Since the property is local, we can suppose that $c$ is the origin of $(\mathbb{C}^2,0)$. Furthermore, since the curve given by $\tilde{f}=0$ is smooth and transverse to the divisor, up to a change of coordinates, we assume that $\{\tilde{f}=0\}=\{y=0\}$ and $\DD=\{x=0\}$. In such coordinates, the curve $\tilde{f}'=0$ admits an equation of the form
$$y+\alpha_{1,0}x+\sum_{i+j\geq 2}\alpha_{i,j}x^iy^j=0.$$ Denote by $h(c)$ the integer $c$ such that $\in \mathcal{C}_{h(c)}(S)$. 
The hypothesis of coincidence of the complete cones implies that for $k\leq N-h(c)=m$ the coefficients $\alpha_{k,0}$ vanish. Hence, we can write
$$y+\alpha_{1,0}x+\sum_{i+j\geq 2}\alpha_{i,j}x^iy^j=y(1+\sum_{i+j\geq 2,j>0}\alpha_{i,j}x^{i}y^{j-1})+x^m(\ldots)=yu(x,y)+x^m(\ldots),$$
where $u$ is a unity. Therefore we have
\[y-\fraction{1}{u(x,y)}(y+\alpha_{1,0}x+\sum_{i+j\geq 2}\alpha_{i,j}x^iy^j)=x^m(\ldots).\]
\end{proof}

\noindent Now, in the neighborhood of any point out of the complete cone, the function $E^*f/E^*f'$ extends holomorphically along the divisor in a non-vanishing function since the multiplicity of $E^*f$ at a generic point of an irreducible component of $\DD$ is a topological invariant \cite{Zariski2}. Hence, applying the previous lemma at any point of the complete cone yields a covering $\mathcal{V}=\left\{ V_i\right\}_{i\in \mathbb{I}}$ of the exceptional divisor and a family $\left\{u_i\right\}_{i\in I},\ u_i\in \mathcal{O}^*(V_i)$ such that
$$E^*f-u_iE^*f'\in\mathcal{O}(-m\DD)(V_i).$$
The $1$-cocycle $\{u_i-u_j\}$ belongs to $\mathcal{Z}^1\left(\DD,\mathcal{O}(-m\DD)\right)$ by decreasing a bit $m$ if necessary by an integer depending only on the multiplicities of $f'$ along the components of the divisor. Now, following \cite{Cam-Mov}, there exists an integer $\delta(m)$ that tends to infinity with $m$ such that the natural cohomological map
$$H^1\left(\DD, \mathcal{O}(-m\DD)\right)\rightarrow H^1\left(\DD, \mathcal{O}(-\delta(m)\DD)\right)$$
is the trivial map. Hence, there exists a $0$-cocycle $\tilde{u}_i$ with values in $\mathcal{O}(-\delta(m)\DD)$ such that
$$u_i-u_j=\tilde{u}_i-\tilde{u}_j.$$
Setting $u=u_i-\tilde{u}_i$, we obtained a global unity $u$ such that
$$E^*f-uE^*f'\in H^0\left(\DD,\mathcal{O}(-m\DD)\right).$$
Blowing down this relation at the origin of $\mathbb{C}^2$ yields 
$$f-\left(E_*u\right)f'\in (x,y)^m.$$
%where $(x,y)$ denotes the maximal ideal at the origin of $\mathbb{C}^2$. 
Following \cite{Mather}, the function $f$ and $E_*Uf'$ are analytically conjugated, which concludes the proof of proposition (\ref{detfin}). Clearly, $N(S)$ only depends on the topological class of $S$ since it does at any step of the proof.
\end{proof}
\bigskip

%\subsection{The complete cone of a curve}

%\noindent Let us consider a sequence of standard blowing up
%$$(\mathcal{M}_n,\DD_n)\xrightarrow{E_n,S_n} (\mathcal{M}_{n-1},\DD_{n-1})\xrightarrow{E_{n-1},S_{n-1}}\ldots \xrightarrow{E_{2},S_2}(\mathcal{M}_1,\DD_1)\xrightarrow{E_1,S_1}(\mathbb{C}^2,0).$$
%Here, $E_i$ refers to the standard blowing-up centered at $S_i$ and $\DD_i$ to the exceptional divisor $(E_1\circ\ldots\circ E_{i})^{-1}(0)$. It is required that $S_i$ is a regular point of $\DD_{i-1}$ in $E_{i-1}^{-1}(S_{i-1})$. The integer $n$ is called the length of the process. The \emph{height} of an irreducible component of the exceptionnal divisor $\DD_n$ is the lowest integer $i$ such that the component appears in $\DD_i$. 

A process of blowing-up $E$ is said to be a \emph{chain process} if, either $E$ is the standard blowing-up of the origin in $\mathbb{C}^2$, or $E=E'\circ E''$ where $E'$ is a chain process and $E''$ is the standard blowing-up of a point that belongs to the smooth part of the heighest irreducible component of $E'$. The lenght of a chain process of blowing-up is the total number of blowing-up and the height of an irreducible component $D$ of the exceptionnal divisor of $E$ is the minimal number of blown-up points so that $D$ appears. A chain process of blowing-up admits \textit{privileged systems of coordinates} $(x,t)$ in a neighborhood of the component of maximal height such that $E$ is written 
\begin{equation}\label{eclatement}
E:(x,t)\mapsto(x, tx^h+t_{h-1}x^{h-1}+t_{h-2}x^{h-2}+\ldots +t_1x).
\end{equation}
The values $t_i$ are the positions of the successive centers in the successive privileged coordinates and $x=0$ is a local equation of the divisor.
   
\noindent Let  $\phi$ be a germ of biholomorphism tangent to the identity map at order $\nu\geq 2$. The function $\phi$ is written 
 $$(x,y)\mapsto (x+A_{\nu}(x,y)+\ldots,\ y+B_{\nu}(x,y)+\ldots )$$
where $A_\nu$ and $B_\nu$ are homogeneous polynomials of degree $\nu$. The following lemma may be proved with an induction on the height of the component:
 
\begin{lemma}\label{actiondiff} The biholomorphism $\phi$ can be lifted-up through any chain process $E$ of blowing-up with a length smaller than $\nu$, i.e. there exists $\widetilde{\phi}$ such that $E\circ\widetilde{\phi}=\phi\circ E$. The action of $\widetilde\phi$ on any component of the divisor of height less than $\nu-1$ is trivial. Its action on any component of height $\nu$ 
%in a process of the form
% $$(x,t)\mapsto(x, tx^i+t_{i-1}x^{i-1}+t_{i-2}x^{i-2}+\ldots +t_1x)$$
is written in privileged coordinates
 $$(0,t)\mapsto \left(0,t+A_\nu(1,t_{1})-t_1 B_\nu(1,t_1)\right)$$ where $t_1$ is the coordinate of the blown-up point on the first component of the irreducible divisor. 
\end{lemma}
\noindent Note that the non trivial action of $\phi$ described above only depends on the position of the first center $t_1$. 

\noindent If $S$ is defined by 
\begin{equation}\label{formnormcurve}
xy(y+x)\prod_{l=1}^{p-3}(y+\sum_{k=1}^{l}a_{k,l}x^k)=0,
\end{equation}
 then the complete cone can be represented in the privileged systems of coordinates by the matrix of dimension $\infty\times p$ 
\begin{eqnarray*}
\left(\begin{array}{cccccccc}
\infty & 0 & 1 & a_{1,1} & a_{1,2} & a_{1,3} & \ldots & a_{1,p-3}\\
0 & 0 & 0 & 0 & a_{2,2} & a_{2,3} & \ldots & a_{2,p-3}\\
0 & 0 & 0 & 0 & 0 & a_{3,3} & \ldots & a_{3,p-3}\\
 &  &  &  & \vdots &  & \ddots & \vdots\\
0 & 0 & 0 & 0 & 0 & 0 & 0 & a_{p-3,p-3}\\
 &  &  &  & \vdots\\
0 & 0 & 0 & 0 & 0 & 0 & 0 & 0\\
 &  &  &  & \vdots\end{array}\right)
\end{eqnarray*}
each line corresponding to one height. Note that the $p-3$ first lines contain nothing but the matrix $a$ itself and that beyond the height $p-2$, any component of the complete cone is a $p$-uple of zeroes.

%\subsection{Normalization of curves}

\noindent We can now prove the main result (\ref{prenorm_curves}) of this section.\bigskip
%\begin{theorem}  Let $S$ be a germ of curve with $p$ irreducible smooth transversal components. Then, up to some change of coordinates, the curve $S$ is given by 
% $$xy(x+y)\prod_{l=1}^{p-3}(y+\sum_{k=1}^{l}a_{kl}x^k)=0, $$
%for some $a$ in $A$.
%\end{theorem}

\begin{proof} Taking the image of $S$ by a suitable linear biholomorphism, we obtain a germ of curve $S$ with a tangent cone of the form $\{\infty,0,1,a_{1,1}\ldots a_{1,p-3}\}$. Now, assume that there exists a germ of biholomorphism $\phi$ of $(\mathbb{C}^2,0)$ such that the complete cone of $\phi^*S$ coincides with the one of a normal form (\ref{formnormcurve}) until height $N$. Denote by $\{t_1,\ldots,t_p\}$ the component of height $N+1$ of the complete cone of $\phi^*S$. Let $Q$ be a polynomial function of degree $N+2$ in one variable such that for any $k\leq \min(p,N+3)$

\begin{eqnarray}
\left. t^{N+2}Q(\fraction{1}{t})\right|_{t=0}+t_1&=&0 \\
Q(a_{1,k})+t_k&=&0\quad \textup{ for any } 2\leq k \leq \min(p,N+3) 
\end{eqnarray}
The first requirement on $Q$ takes care of the special position of the curve $\{x=0\}$ with respect to our choice of privileged coordinates. The polynomial function $Q$ does exist since the numbers $a_{1,i}$ satisfy $a_{1,i}\neq a_{1,j}$ for $i\neq j$.
Take any decomposition of $Q(t)$ of the form $a(t)-tb(t)$ where $a$ and $b$ have degree at most $N+1$ and let $A_{N+1}$ and $B_{N+1}$ be the homogeneous polynomial functions of degree $N+1$ such that
$A_{N+1}(1,t)=a(t) \textup{  and  } B_{N+1}(1,t)=b(t)$. 
Finally, let $\phi_{N+1}$ be the biholomorphism defined by
$$\phi_{N+1}(x,y)=(x+A_{N+1}(x,y),y+B_{N+1}(x,y)).$$
According to lemma (\ref{actiondiff}), the complete cone of $\phi_{N+1}^*\phi^*S$ is the one of $\phi^*S$ until height $N$ and its component of height $N+1$ contains at least $N+3$ zeroes. Hence, the complete cone of  $\phi_{N+1}^*\phi^*S$ coincides with that of a normal form until height $N+1$. By induction, the proposition is a consequence of the finite determinacy statement.
\end{proof}

\noindent The next result deals with the unicity of normal forms for the separatrices. A germ of biholomorphism $\phi$
$$(x,y)\rightarrow \left(x+A_\nu(x,y)+\ldots,\ y+B_\nu(x,y)+\ldots\right)$$ 
is said to be \emph{dicritical} if $xA_\nu(x,y)-yB_\nu(x,y)$ vanishes. 

\begin{prop}
Let $S_a$ and $S_b$ be two germs of curves given by normal forms $N_a^{(1)}$ and $N_b^{(1)}$ with $a,b\in A$, which are conjugated by a non-dicritical biholomorphism. Then $a$ and  $b$ are equal up to permissible permutation of the columns.
\end{prop}
\noindent By permissible permutations, we mean permutations preserving the triangular profile of the matrix of parameters.
\begin{proof} 
Changing $b$ into $\lambda\cdot b$ with $\lambda\in \mathbb{C}^*$, we can suppose that $\phi$ is tangent to the identity map for a certain order $N$. If $N$ is bigger than $p-2$ then the complete cone of $S_a$ and $S_b$ coincide until height $p-3$. Hence, $a$ and $b$ are equal up to permissible permutation of the columns. If $N$ is smaller than $p-3$, let us write
$$\phi(x,y)=(x+A_N(x,y)+\ldots,\ y+B_N(x,y)+\ldots).$$
Since the action of $\phi$ on any component of height $N$ conjugates the complete cones, the function $A_N(1,t)-tB_N(1,t)$ vanishes on
 $\{\infty,0,1,a_{1,1},\ldots,a_{1,p-3}\}$, which is the common tangent cone of $S_a$ and $S_b$. Since the degree of $A_N(1,t)-tB_N(1,t)$ is at most $N+1$, it is the zero polynome. Hence,
$$xA_N(x,y)-yB_N(x,y)=0,$$
which is impossible since $\phi$ is non-dicritical.
\end{proof}
\noindent One cannot remove the assumption of non-dicriticalness in the previous result: for example, consider the curve given by 
$$xy(y+x)(y+a_{1,1}x)(y+a_{1,2}x)(y+a_{1,3}x+a_{2,3}x^2)=0.$$
The germ of dicritical biholomorphism $(x,y)\rightarrow (1+x)(x,y)$ conjugates the previous curve and the following one 
$$xy(y+x)(y+a_{1,1}x)(y+a_{1,2}x)(y+a_{1,3}x+a_{2,3}x^2+a_{2,3}x^3)=0.$$Clearly, their cones are not equivalent up to permissible permutations.

\section{The local moduli space of foliations.}

This section is devoted to the proof of the following result:
\begin{theorem}\label{thm_local} 
We fix an element $a^0$ of $A$ and we consider the following germ of family of functions:
 $$N_a^{(n)}=x^{n_1}y^{n_2}(y+x)^{n_3}\prod_{l=1}^{p-3}(y+\sum_{k=1}^{l}a_{k,l}x^k)^{n_{l+3}},\ a\in(A,a^0).$$
This family defines a semi-universal equireducible unfolding of the foliation $\FF_0$ defined by $N_{a^0}^{(n)}$.
\end{theorem}

\noindent This means that for any equireducible unfolding $F_p$, $p\in (P,p^0)$ which defines $\FF_0$ for $p=p^0$, there exists a map $\lambda: P\rightarrow A$ such that the family $F_{p}$ is analytically equivalent to $N_{\lambda(t)}$. Furthermore, the derivative of $\lambda$ at $p^0$ is unique.
\bigskip

% Il faudra d�finir toutes les relations d'�quivalences de fonctions, courbes et feuilletages, locales et globales.

\noindent We consider the blowing-up of the origin $E: (\MM,\DD) \rightarrow (\mathbb{C}^2,0)$ with its exceptional divisor $\DD=E^{-1}(0).$ The manifold $\MM$ is defined by the two charts $U_1,(x_1,y_1)$, $U_2,(x_2,y_2)$ in which $E(x_1,y_1)=(x_1,x_1y_1)$, $E(x_2,y_2)=(x_2y_2,y_2)$. The change of coordinates is given by $(x_2=1/y_1, y_2=x_1y_1).$
According to the choice of a generic tangent cone, this blowing up desingularizes each foliation defined $\widetilde{N_a}=E^{-1}(N_a)$. Recall that (see \cite{Mattei2}):
\begin{enumerate}

\item After desingularization by the blowing-up, any unfolding $\widetilde{\FF_p}, p\in(P,p^0)$ of $\widetilde{\FF_0}$ is locally analytically trivial. Indeed, around any regular or singular point $m$, one can consider the one form defined in the first chart by
$$d\widetilde{F}:=\frac{\partial \widetilde{F_p}}{\partial x_1}dx_1+\frac{\partial \widetilde{F_p}}{\partial y_1}dy_1+\sum_i\frac{\partial \widetilde{F_p}}{\partial p_i}dp_i.$$
The codimension one foliation on $\MM\times P$ defined by this integrable one-form is an unfolding which ''contains'' the family $\widetilde{F_p}, p\in(P,p^0)$. Then, for each parameter $p_i$, setting $$X_i=\alpha_i(x_1,y_1,p)\frac{\partial}{\partial x_1} + \beta_i(x_1,y_1,p)\frac{\partial}{\partial y_1}+\frac{\partial}{\partial p_i},$$ one can solve on a neighbourhood $U$ of $m$ in $\MM\times P$ the equation 
$d\widetilde{F}(X_i)=0$ or equivalently
\begin{eqnarray}\label{triv_loc}
	\frac{\partial \widetilde{N_p}}{\partial p_i}=\alpha_i(x_1,y_1,p)\frac{\partial \widetilde{N_p}}{\partial x_1}
	+\beta_i(x_1,y_1,p)\frac{\partial \widetilde{N_p}}{\partial y_1}.
\end{eqnarray}
The local trivialization $\varphi_U$ on ${U}$ is obtained by successive integrations of these vector fields $X_i$.

\item The set of the classes of unfoldings $\widetilde{\FF_p}$ of $\widetilde{\FF_0}$ with parameter $p$ in $(P,p^0)$ up to analytic equivalence is in bijection with the first non abelian cohomology group $H^1(D,G_P)$ where $G_P$ is the sheaf on $D$ of the germs of automorphisms of the trivial deformation on $\MM\times P$ which commute with the projection on $P$, and are equal to the identity on the divisor. This map is defined by the cocycle $\{\varphi_{U,V}\}$ induced by the local trivializations $\varphi_U$ previously obtained.

\item For any Stein open set $U$ in $D$ we have $H^1(U,G_P)=0$ and therefore $H^1(D,G_P)=H^1(\UU,G_P)$ where $\UU$ is the covering of $D$ by the two domains $U_1$, $U_2$.

\item Let $\Theta_0$ be the sheaf on $D$ of germs of holomorphic vector fields in $\mathcal{M}$ tangent to the foliation $\FF_0$. For each direction defined by $v$ in $T_{p^0}P$, the derivative of $\{\varphi_{U,V}\}$ in this direction defines a map from $T_{p^0}P$ into $H^1(D,\Theta_0)$. We denote this map by
%$\frac{\partial \widetilde{\FF_p}}{\partial p}|_{p=p^0}$.
$d\widetilde{\FF_p}(p^0)$.

\noindent If $X_U$ is a collection of local vector fields solutions of (\ref{triv_loc}), the cocycle $\{X_{U,V}=X_U-X_V\}$ evaluated at $p=p^0$ is the image of the direction $\partial/\partial p_i$ in $H^1(D,\Theta_0)$ by $d\widetilde{\FF_p}(p^0)$.

\item If $\theta_0$ is a holomorphic vector field with isolated singularities which defines $\widetilde{F_0}$ on $\mathcal M$, by writing each cocycle as a product of $\theta_0$ with a meromorphic function, we obtain the following identification
 $$H^1(D,\Theta_0)\simeq\oplus_{k-p+1<l<0, k\geq 1}\C x_1^{k-1}y_1^{l}\theta_0.$$ In particular, the dimension of this vector space is $\delta=(p-2)(p-3)/2$. Hereafter, we shall denote by $\left[\fraction{x_1^{k-1}}{y_1^{-l}}\right]$ the class of $x_1^{k-1}y_1^{l}\theta_0$ in $H^1(D,\Theta_0)$.
\end{enumerate}

\noindent Now we shall make use of the following theorem (\cite{Mattei2}, theorem (3.2.1)):

\begin{theorem} The unfolding $\FF_p, p\in(P,p^0)$ is semi-universal among the equireducible unfoldings of $\FF_0$ if and only if the map $d\widetilde{\FF_p}(p^0):\ T_{p^0}P\longrightarrow H^1(D,\Theta_0)$ is a bijective mapping.
\end{theorem}

\noindent Note that the dimension $\delta$ of the vector space $H^1(D,\Theta_0)$ is also the dimension of the triangular parameter space $A$. 
Therefore the proof of theorem (\ref{thm_local}) reduces to the 

\begin{prop} We consider the unfolding $\widetilde{\FF_a}$ defined by the blowing up of $N^{(n)}_a$, $a\in(A,a^0)$.
The images of the $\frac{\partial}{\partial a_{k,l}}$ in $H^1(D,\Theta_0)$ by $d\widetilde{\FF_a}(a^0)$ are linearly independent.
\end{prop}

\begin{proof}
In the first chart, we have
$$\widetilde{N}^{(n)}_a(x_1,y_1)=x_1^{|n|}y_1^{n_2}(y_1+1)^{n_3}\prod_{l=1}^{p-3}(y_1+\sum_{k=1}^{l}a_{k,l}x_1^{k-1})^{n_{l+3}}$$
Note that:
$$\frac{\partial \widetilde{N}^{(n)}_a}{\partial a_{k,l}} = \widetilde{N}^{(n)}_a n_{l+3}\frac{x_1^{k-1}}{y_1+\sum_{i=1}^l a_{i,l}x_1^{i-1}}=x_1^{k-1}\frac{\partial \widetilde{N}^{(n)}_a}{\partial a_{1,l}}.$$
Therefore, if we have a solution $X_{k,l}=\alpha_{k,l}\frac{\partial}{\partial x_1} + \beta_{k,l}\frac{\partial}{\partial y_1}+\frac{\partial}{\partial a_{k,l}}$ which satisfy
\begin{eqnarray}\label{triv_loc_kl}
	\frac{\partial \widetilde{N}^{(n)}_{a}}{\partial a_{k,l}}=\alpha_{k,l}\frac{\partial \widetilde{N}^{(n)}_a}{\partial x_1}
	+\beta_{k,l}\frac{\partial \widetilde{N}^{(n)}_a}{\partial y_1}
\end{eqnarray}
for $k=1$, then we obtain a solution for the other values of $k$ setting: $$X_{k,l}=x_1^{k-1}X_{1,l}.$$
Now we solve (\ref{triv_loc_kl}) for $k=1$. We have:
\begin{eqnarray*}
	\frac{\partial \widetilde{N}^{(n)}_{a}}{\partial a_{1,l}}&=&n_{l+3}x_1^{|n|}y_1^{n_2}(y_1+1)^{n_3}	\frac{\prod_{j=1}^{p-3}(y_1+\sum_{i=1}^{l}a_{i,l}x_1^{i-1})^{n_{j+3}}}{y_1+\sum_{i=1}^{l}a_{i,l}x_1^{i-1}}\\
	&=&n_{l+3}x_1^{|n|}\left(\frac{P(y_1)}{y_1+a_{1,l}}+x_1(\ldots)\right)
	\end{eqnarray*}
	with $P(y_1)=y_1^{n_2}(y_1+1)^{n_3}\prod_{j=1}^{p-3}(y_1+a_{1,j})^{n_{j+3}}.$
Now, since $\widetilde{N}^{(n)}_{a}$ is equal to $x_1^{|n|}\left(P(y_1)+x_1(\ldots)\right)$, we have
\begin{eqnarray}
	\frac{\partial \widetilde{N}^{(n)}_{a}}{\partial x_1}&=&|n|x_1^{|n|-1}P(y_1)+x_1^{|n|}(\ldots),\label{forder}\\
	\frac{\partial \widetilde{N}^{(n)}_{a}}{\partial y_1}&=&x_1^{|n|}P'(y_1)+x_1^{|n|+1}(\ldots).
\end{eqnarray}
Setting $\alpha_{1,l}=x_1\widetilde{\alpha_{1,l}}$, we deduce from (\ref{triv_loc_kl}) that
\begin{eqnarray}
	n_{l+3}\frac{P(y_1)}{y_1+a_{1,l}}&=&|n|\widetilde{\alpha_{1,l}}P(y_1)+
	\beta_{1,l}P'(y_1)+x_1(\ldots).
\end{eqnarray}
By using B\'ezout identity, there exist polynomials $U$ and $V$ in $y_1$ such that
$$P\wedge P'=UP'+VP$$
where $P\wedge P'$ is the great common divisor of $P$ and $P'$. Dividing $U$ by $P$ if necessary, the polynomial function $U$ may be chosen of degree $p-2$. Let us denote by $R$ the polynomial function satisfying
$$P=(P\wedge P')R.$$ 
It is written $R(y_1)=y_1(y_1+1)\prod_{j=1}^{p-3}(y_1+a_{1,j})$,
and we obtain a solution of (\ref{triv_loc_kl}) in the first chart:
\begin{eqnarray*}
	\alpha_{1,l}&=&x_1\frac{n_{l+3}}{|n|}V(y_1)\frac{R(y_1)}{y_1+a_{1,l}}+x_1^2(\ldots)\\
	\beta_{1,l}&=&n_{l+3}U(y_1)\frac{R(y_1)}{y_1+a_{1,l}}+x_1(\ldots)\\
	\mbox{i. e. }X_{1,l}^{(1)}&=&n_{l+3}U(y_1)\frac{R(y_1)}{y_1+a_{1,l}}\frac{\partial }{\partial y_1}+x_1(\ldots).
\end{eqnarray*}
Similarly, in the second chart we write
$$\widetilde{N_{a}}(x_2,y_2)=y_2^{|n|}\big(Q(x_2)+y_2(\ldots)\big)$$
with $$Q(x_2)=x_2^{n_1}(x_2+1)^{n_3}\prod_{j=1}^{p-3}(1+a_{1,j}x_2)^{n_{j+3}}.$$
We set $Q\wedge Q'=WQ'+ZQ$ and $Q=(Q\wedge Q')S$ with $S=x_2(x_2+1)\prod_{j=1}^{p-3}(1+a_{1,j}x_2)$. As before, we can assume that the degree of $W$ is $p-2$. We obtain the solution 
$$X_{1,l}^{(2)}=n_{l+3}W(x_2)\frac{S(x_2)}{1+a_{1,l}x_2}\frac{\partial }{\partial x_2}+y_2(\ldots).$$
To compute the cocycle we write $X_{1,l}^{(2)}$ in the first chart, we use the standard change of coordinates $x_1=y_2x_2$ and $y_1=\frac{1}{x_2}$. Since we have
$$ W(x_2)=\frac{\tilde{W}(y_1)}{y_1^{p-2}}\ \mbox{and } S(x_2)=\frac{R(y_1)}{y_1^{p}},$$ where $\tilde{W}$ is a polynomial function, 
we finally find the first term of the cocycle 
\begin{eqnarray*}
X_{1,l}^{(1,2)} & = & X_{1,l}^{(1)}-X_{1,l}^{(2)}=n_{l+3}\frac{R(y_{1})}{y_{1}+a_{1,l}}\left( U(y_{1})+\tilde{W}(y_{1})y_{1}^{5-2p}\right)\frac{\partial}{\partial y_{1}}+x_{1}(\ldots)
\end{eqnarray*}
If $\theta_0$ is a holomorphic vector field with isolated singularities which defines $\widetilde{\mathcal{F}}_0$ on $U_1\cap U_2$, we can write
$$X_{1,l}^{(1,2)}=\Phi_{1,l}^{(1,2)}\theta_0.$$
Actually, one can choose $\theta_0=\frac{1}{x_1^{|n|-2}}E^*\left(\frac{\partial N^{(n)}_a}{\partial x}\frac{\partial }{\partial y}-\frac{\partial N^{(n)}_a}{\partial y}\frac{\partial}{\partial x}\right)$.
Following the point (5), the set of the coefficients of the Laurent's series of $\Phi_{1,l}^{(1,2)}$ characterizes the class of $X_{1,l}^{(1,2)}$ in $H^1(D,\Theta_0)$. Now, according to (\ref{forder}), 
\begin{eqnarray*}
\Phi_{1,l}^{(1,2)} & = & \frac{n_{l+3}}{\left|n\right|}\frac{R(y_{1})}{P(y_{1})}\frac{1}{y_{1}+a_{1,l}}\left(U(y_{1})+\tilde{W}(y_{1})y_{1}^{5-2p}\right)+x_1(\cdots)\\
 & = & \frac{n_{l+3}}{y_{1}^{2p-6+n_{2}}}\frac{f\left(y_{1}\right)}{y_{1}+a_{1,l}}+x_1(\cdots)\end{eqnarray*}
where $f$ is a meromorphic function independent of $l$, holomorphic in a neighborhood of $0$ with $f(0)\neq 0$.

\begin{lemma}\label{independants}  For any $d\leq p-3$, the family of functions $$\left\{\frac{1}{y_{1}^{2p-6+n_{2}}}\frac{f\left(y_{1}\right)}{y_{1}+a_{1,l}}\right\}_{l=d\ldots p-3}$$ is a free family in the quotient space of the space of Laurent series in $y_1$ by the subspace 
$\mathbb{C}\{y_1\}\oplus\textup{Vect}\left\{\frac{1}{y_1^k}\right\}_{k\geq p-3-d}$
\end{lemma}

\begin{proof} The lemma is true if $f(y_1)=1$. Indeed, the meromorphic functions $\frac{n_{l+3}}{y_1+a_{1,l}}$ are conjugated (up to multiplicative constants) to the first one by distinct non trivial linear transformations $y_1\mapsto \lambda y_1$ which acts on the coefficients of their Laurent's series by $a_k\mapsto\lambda^ka_k$. Therefore the indepence comes from the maximal rank of Vandermonde matrices. Moreover, the multiplication by a non-vanishing function $f$ induces an inversible linear map: its matrix is a triangular one with non vanishing entries on its diagonal. Hence, the lemma is proved.
\end{proof}

\noindent From this lemma and the previous description of $H^1(D,\Theta_0)$, we deduce the linear independence of $X_{1,l}^{(1,2)}(0,y_1)$ in $H^1(D,\Theta_0)$. Suppose now that we have a linear relation between the cocycles $X_{k,l}^{(1,2)}=\Phi_{k,l}^{(1,2)}\cdot \theta_0=x_1^{k-1}\Phi_{1,l}^{(1,2)}\cdot \theta_0$:
$$\sum_{1\leq k,l\leq p-3, k\leq l}\lambda_{k,l}x_1^{k-1}\Phi_{1,l}^{(1,2)}\cdot \theta_0=0$$
in $H^1(D,\Theta_0)$. Evaluating at $x_1=0$ we get from lemma (\ref{independants}) that $\lambda_{1,l}=0$ for all $l$. Now we can divide the relation by $x_1$, and iterate the argument. Finally, the cocycles related to the directions $\frac{\partial}{\partial a_{k,l}}$ are independent in $H^1(D,\Theta_0)$.
\end{proof}

\noindent We shall need the following additional information on this semi-universal space:

\begin{prop}\label{polynomial}
The coefficient of $\frac{\partial}{\partial a_{k,l}}$
in the basis $\left\{ \left[\frac{x_{1}^{i-1}}{y_{1}^{j}}\right]\right\} _{{\scriptsize{\begin{array}{c}
1\leq i\leq p-3\\
1\leq j\leq p-2-i\end{array}}}}$ are in the ring $\mathbb{C}\left(a_{1}\right)\left[a_{2},\ldots,a_{p-3}\right]$, where $a_i$ stands for $\{a_{i,l},\ l=i,\cdots p-3\}$.\end{prop}

\begin{proof}
The equation (\ref{triv_loc_kl}) can be solved in the following way:
looking at the homogeneous part of order $\nu$ yields\[
J_{\nu}\left(\frac{\partial\widetilde{N}_{a}^{\left(n\right)}}{\partial a_{k,l}}\right)=\sum_{i+j=\nu}J_{i}\left(\alpha_{k,l}\right)J_{j}\left(\frac{\partial\widetilde{N}_{a}^{\left(n\right)}}{\partial x_{1}}\right)+J_{i}\left(\beta_{k,l}\right)J_{j}\left(\frac{\partial\widetilde{N}_{a}^{\left(n\right)}}{\partial y_{1}}\right).\]
Hence, we find the following induction relation \[
\begin{array}{l}
J_{\nu-\left|n\right|}\left(\alpha_{k,l}\right)J_{\left|n\right|}\left(\frac{\partial\widetilde{N}_{a}^{\left(n\right)}}{\partial x_{1}}\right)+J_{\nu-\left|n\right|}\left(\beta_{k,l}\right)J_{\left|n\right|}\left(\frac{\partial\widetilde{N}_{a}^{\left(n\right)}}{\partial y_{1}}\right)\\
=J_{\nu}\left(\frac{\partial\widetilde{N}_{a}^{\left(n\right)}}{\partial a_{k,l}}\right)-\displaystyle\sum_{\scriptsize\begin{array}{l}
i+j=\nu\\
j\neq\left|n\right|\end{array}}J_{i}\left(\alpha_{k,l}\right)J_{j}\left(\frac{\partial\widetilde{N}_{a}^{\left(n\right)}}{\partial x_{1}}\right)+J_{i}\left(\beta_{k,l}\right)J_{j}\left(\frac{\partial\widetilde{N}_{a}^{\left(n\right)}}{\partial y_{1}}\right)\end{array}\]
Now, the coefficients of $J_{\left|n\right|}\left(\frac{\partial\widetilde{N}_{a}^{\left(n\right)}}{\partial x_{1}}\right)$
and $J_{\left|n\right|}\left(\frac{\partial\widetilde{N}_{a}^{\left(n\right)}}{\partial y1}\right)$
depend only on the variables $a_{1}$. Moreover the coeffcients
of $J_{j}\left(\frac{\partial\widetilde{N}_{a}^{\left(n\right)}}{\partial x_{1}}\right)$
and $J_{j}\left(\frac{\partial\widetilde{N}_{a}^{\left(n\right)}}{\partial y_{1}}\right)$
are polynomial. Hence, an induction
on $\nu$ ensures that for all $\nu$ the coefficients of $J_{\nu}\left(\alpha_{kl}\right)$
and $J_{\nu}\left(\beta_{kl}\right)$ can be chosen rational in $a_{1}$
and polynomial in the variables $a_{k}$, $k\geq2$. The same
result holds for the relation (\ref{triv_loc_kl}) in the second chart. Now, following
the computation of the cocycle in the previous proof make it
obvious that the coefficients in its Laurent development are in $\mathbb{C}\left(a_{1}\right)\left[a_{2},\ldots,a_{p-3}\right]$.
So are the coordinates of $\frac{\partial}{\partial a_{k,l}}$
in the standard basis.\end{proof}

\section{The global moduli space for foliations.}

This section is devoted to the proof of the following theorem

\begin{theorem}
For any $f$ in $\mathcal{T}^{(n)}$, there exists a unique $a\in A$ up to the action of $\mathbb{C}^*$ on $A$ defined in the introduction such that
$$f\sim N_a^{(n)}$$
the conjugacy preserving the numbering of the branches.
\end{theorem}
\noindent Note that to remove the numbering property may complicate the situation: indeed, in the case of four irreducible components, the normal forms are $xy(y+x)(y+a_{1,1}x)$ with $a_{1,1}\in \mathbb{C}\backslash \{0,1\}$. Now, the biholomorphism $(x,y)\mapsto (y,x)$, which cannot preserve any numbering of the irreducible components, conjugates the normal forms associated to the parameter $a_{1,1}$ and $1/a_{1,1}$. Therefore, whereas the marked moduli space is simply $\mathbb{C}\backslash\{0,1\}$, the \emph{free} moduli space appears to be $\mathbb{C}\backslash\{0,1\}$ quotiented by the finite group of order $6$ of automorphisms of  $\mathbb{C}\backslash\{0,1\}$
$$\left\{z\mapsto z,z\mapsto \frac{1}{z},z\mapsto 1-z,z\mapsto 1-\frac{1}{z},z\mapsto \frac{1}{1-z},z\mapsto \frac{z}{z-1}\right\}. 
$$ which is not a smooth manifold.

\begin{prop}[Existence of normal forms]\label{existence}
For any $f$ in $\mathcal{T}^{(n)}$, there exists $a\in A$ such that
$f\sim N_a^{(n)}$.
\end{prop}

\begin{proof}
Let us consider any foliation given by the level of a reduced function desingularized after one blowing-up, given by a function of the form
$$f_1^{n_1} f_2^{n_2} \ldots f_p^{n_p}$$
where the $f_i$'s are irreducible functions such that $f_i(0)=0$ and $f'_i(0)\neq f_j'(0)\neq 0$ for $i\neq j$. From the results of the first section, once the curve $f_1 f_2 \ldots f_p=0$ is normalized, the foliation is given by a function 
$$u(x,y)N^{(n)}_{a^0}(x,y)$$
for a $a^0\in A$ and $u(0,0)=1$. A classical statement ensures that $u$ can be supposed to be polynomial \cite{Mather}: roughly speaking, for $N$ big enough depending on the Milnor number of $f_1f_2\ldots f_p$ the deformation
$$\big( {J}_N(u)+t(u-{J}_N(u))\big)N^{(n)}_{a^0}(x,y),$$ where ${J}_N$ refers to the jet of order $N$, is analytically trivial.
Let us consider the decomposition of $u$ in homogeneous components
$$u=1+\sum_{i=1}^M u_i(x,y)$$ and
 the deformation with parameter in $\mathbb{C}^M\times A$ defined by
$$F_{t_1,\ldots,t_M,a}(x,y)=\left(1+\sum_{i=1}^M t_iu_i(x,y)\right)N^{(n)}_a(x,y).$$
Along the whole space of parameters, this deformation is an equisingular unfolding of foliation.
Let $\overline{a^0}$ be the zero matrix except on the first line where it coincides with $a^0$.
The following relation holds
$$F_{1,\ldots,1,a^0}=uN^{(n)}_{a^0}\qquad F_{0,\ldots,0,\overline{a^0}}=N^{(n)}_{\overline{a^0}}.$$
Note that $N^{(n)}_{\overline{a^0}}$ is an homogeneous function which is equal the homogeneous component of smallest degree of $N^{(n)}_{a^0}$. Now, the essential fact concerning our normal forms is the following functional relation: let $\lambda$ in $\mathbb{C}^*$ then
$$N_a^{(n)}(\lambda x,\lambda y)=\lambda^{|n|}\cdot N_{\lambda\cdot a}^{(n)}.$$ 
The deformation $F$ satisfies the same kind of functional equation
$$F_{t_1,\ldots,t_M,a^0}(\lambda x,\lambda y)=\lambda^{|n|} F_{\lambda t_1,\ldots,\lambda^M t_M,\lambda\cdot a^0}(x,y).$$
In particular, the foliations given by the functions  $F_{t_1,\ldots,t_M,a^0}$ and $F_{\lambda t_1,\ldots,\lambda^M t_M,\lambda\cdot a^0}$ are analytically equivalent. Moreover, when $\lambda$ tends to zero then $\lambda\cdot a$ tends to $\overline{a^0}$. Hence, according to the previous local result, there exist $\lambda\in \mathbb{C}^*$ and $\tilde{a}\in A$ such that the foliations given by $F_{\lambda ,\ldots,\lambda^M,\lambda\cdot a^0}$ and $N^{(n)}_{\tilde{a}}$ are analytically conjugated. Therefore, the same conclusion holds for $uN^{(n)}_{a^0}$ and $N^{(n)}_{\tilde{a}}$. This concludes the proof of the  existence of normal forms for each $f$ in $\mathcal{T}^{(n)}$.
\end{proof}

\begin{prop}[Unicity of normal forms]\label{unicity} The foliations defined by $N^{(n)}_a$ and $N^{(n)}_b$, $a$, $b$ in $A$, are equivalent if and only if there exists $\lambda$ in $\C^*$ such that $\lambda\cdot a=b$.
\end{prop}

\begin{proof}
Let us suppose that two foliations given by normal forms $N^{(n)}_a$ and $N^{(n)}_b$, $a$, $b$ in $A$, are analytically conjugated by a conjugacy preserving the numbering. There exist a germ of biholomorphism $\phi$ of $(\mathbb{C}^2,0)$ and $\psi$ a germ of biholomorphism of $(\mathbb{C},0)$ such that
$$N^{(n)}_a\circ \phi =\psi\circ N^{(n)}_b$$
According to \cite{Cer-Ber}, one can find a germ $\tilde{\phi}$ of biholomorphism such that
$\psi\circ N^{(n)}_b=\psi'(0) N^{(n)}_b\circ \tilde{\phi},$
which yields
$$N^{(n)}_a\circ \phi\circ\tilde{\phi}^{-1} =\psi'(0)N^{(n)}_b.$$
Composing on the right side with a suitable homothety gives us a biholomorphism $\Phi$ tangent to the identity and two non-vanishing complex numbers $\eta$ and $\lambda$ satisfying
\begin{equation}\label{compo}
N^{(n)}_a\circ\Phi=\lambda N^{(n)}_{\eta\cdot b}.
\end{equation}
Taking a close look to the homogeneous part of each above terms yields $\lambda=1$. 

\medskip
\noindent If $\Phi$ is non-dicritical or tangent to Id with an order bigger than $p-2$, then, according to the results of the first part, $a$ and $b$ are equal. Hence, from now on, we suppose $\Phi$ to be dicritical and tangent to the identity with an order $\nu+1$ smaller than $p-3$. For the sake of simplicity, we keep on denoting the matrix $\eta\cdot b$ by $b$. Since $\Phi$ is tangent to the identity, it is the time one of the flow of a formal dicritical vector field
$$\Phi=e^{\hat{X}}.$$
Its homogeneous part of degree $\nu+1$ is radial and is written
$\phi_\nu R$ where $\phi_\nu$ stands for an homogeous polynomial function of degree $\nu$ and $R$ for the radial vector $x\partial_x+y\partial_y$.
The relation (\ref{compo}) can be expressed as follows 
\begin{equation}\label{equa}
{e^{\hat{X}}}^*N^{(n)}_a=N_a^{(n)}+\phi_\nu R\cdot N_a^{(n)}+\cdots=  N^{(n)}_b
\end{equation}
In this relation, the valuation of $\phi_\nu R\cdot N_a^{(n)}$ is at least $\nu+|n|$. Since $\Phi$ is tangent to identity at order $\nu+1$, the complete cones of $N^{(n)}_a$ and $N^{(n)}_b$ coincide until height $\nu$. Now, for any $k$, it is easily seen that the $k$-jet of $N^{(n)}_a$ only depends on the components of the complete cone of height lower than $k-|n|+1$. Hence, the terms with valuation lower than $\nu+|n|-1$ in (\ref{equa}) are equal. Therefore, the first non-trivial homogeneous part of the relation (\ref{equa}) is of valuation $\nu+|n|$ and is written
$$
 N^{(n)}_{a,\nu+|n|} +\phi_\nu R\cdot N_{a,|n|}^{(n)}=N_{b,\nu+|n|}^{(n)},
$$
where $N_{a,k}^{(n)}$ stands for the homogeneous part of degree $k$ in $N ^{(n)}_a$.
Since $N_{a,|n|}^{(n)}$ is homogeneous, this relation becomes
\begin{equation*}
 N^{(n)}_{a,\nu+|n|}-N_{b,\nu+|n|}^{(n)} + |n|\phi_\nu N_{a,|n|}^{(n)}=0.
\end{equation*}
The homogeneous component of degree $\nu+|n|$ in $N^{(n)}_a$ is written  
\begin{equation*}\sum_{i=\nu+1}^{p-3} a_{{\nu+1},i}x^{\nu+1}\fraction{{N^{(n)}_{a,|n|}}}{y+a_{1,i}x}+H_a(x,y).\end{equation*}
Here, $H_a$ is a polynomial function whose coefficients only depend on the component of the complete cone of height lower than $\nu$. Since $a_{1,i}=b_{1,i}$, the difference ${N^{(n)}_{b,\nu+|n|}}-{N^{(n)}_{a,\nu+|n|}}$ is simply written
$${N^{(n)}_{a,|n|}} \sum_{i=\nu+1}^{p-3} \fraction{\lambda_ix^{\nu+1}}{y+a_{1,i}x}.$$
where $\lambda_i=b_{{\nu+1},i}-a_{{\nu+1},i}$. Therefore, the polynomial function $\phi_\nu$ must coincide with
\[-\frac{1}{|n|}\sum_{i=\nu+1}^{p-3} \fraction{\lambda_ix^{\nu+1}}{y+a_{1,i}x}\]
which happens to be a polynome if and only if $\lambda_i$ vanishes for all $i$. Therefore the complete cones coincide until height $\nu+1$. 
\end{proof}

\noindent We denote by $\mathbb{P}A$ the quotient of $A$ by the weighted action of $\mathbb{C}^*$ defined in the introduction. 
This space is a fiber space over $\C^{p-3}\backslash\Delta$, where $\Delta$ is defined by the conditions $a_{1,j}\neq a_{1,k}\neq 0,1$ for $j\neq k$. The fiber is a weighted projective space and thus is simply connected.
The previous theorem yields the following result 
\begin{theorem} 
The marked moduli space of foliations $M^{(n)}$ defined by the functions $f$ in $\mathcal{T}^{(n)}$, punctured by the class of the homogeneous foliation, is isomorphic to the weighted projective space $\mathbb{P}A$. Its dimension is $\frac{(p-2)(p-3)}{2}-1$, where $p$ is the number of branches of $f$.
\end{theorem}

\section{Extension to the Darboux functions.}

In this section, we shortly show how the previous result can be extended
to the case of the Darboux functions:\[
f^{(\lambda)}=f_{1}^{\lambda_{1}}\cdots f_{p}^{\lambda_{p}}\]
where the $\lambda_{i}$'s are fixed complex numbers. In this case,
the foliation is defined by the holomorphic one form $$\omega=f_1\cdots f_p\sum_{i=1}^p\lambda_i\frac{df_i}{f_i},$$ and is called \emph{logarithmic}, since it is defined by a closed logarithmic one-form outside the curve $S:\ f_1\cdots f_p=0$. We extend the definition of the equivalence relation $\sim$ defined in the introduction to the Darboux case setting: $f_0\sim f_1$ if and only if the two corresponding foliations defined by $\omega_0$ and $\omega_1$ are equivalent: there exists $\phi$ in  $\mbox{Diff}\ (\mathbb{C}^2,0)$ such that $\phi^*\omega_0\wedge\omega_1=0$. We suppose that:\smallskip

(i) the irreducible branches $f_i=0$ are smooth, with distinct tangencies;

(ii) $|\lambda|=\sum_i \lambda_i\neq 0$;

(iii) the projective p-uple $(\lambda_1 : \ldots : \lambda_p)$ is distinct from $(n_1:\ldots :n_p)$ for any $(n)$ in $\N^p$.\smallskip

\noindent The first assumption insures that the foliation defined by $f^{(\lambda)}$ is desingularized after one blowing-up. The second one implies that the desingularized foliation is non dicritical: the exceptional divisor is invariant. The last one prevents the foliation from admitting a uniform first integral.
We denote by $\mathcal{D}^{(\lambda)}$ the set of the Darboux functions with multiplicities $(\lambda)$ satisfying the three previous conditions.

\begin{theorem}
Let $f^{(\lambda)}$ in $\mathcal{D}^{(\lambda)}$.
There exists a unique $a\in A$ up to the action of $\mathbb{C}^*$ such that $ f\sim N_a^{(\lambda)}$,
the conjugacy preserving the numbering of the branches.
\end{theorem}

\begin{proof} The scheme of the proof is exactly the same as the uniform case one: we first obtain a local result, proving that $N_a^{(\lambda)}$, $a\in (A,a_0)$ is a semi-universal unfolding of the foliation defined by $N_{a_0}^{(\lambda)}$. Next, this local result is globalized thanks to a functional relation satisfied by the normal forms. Finally we discuss the unicity of these normal forms.
We consider
$$ \Omega(x,y,a)=N_{a}^{(1)}\left(\lambda_{1}\frac{dx}{x}+\lambda_{2}\frac{dy}{y}+\lambda_{3}\frac{d(y+x)}{y+x}+\sum_{i=1}^{p-3}\lambda_{i+3}\frac{dP_{i}}{P_{i}}
+\sum_{i=1}^{p-3}\frac{\lambda_{i+3}}{P_i}\sum_{k=1}^{i}da_{k,i}x^{k}\right)
$$
where $P_i=y+\sum_{k=1}^i a_{k,i} x^k$.
\noindent We claim that in the neighborhood of any $a_0\in A$, this is a semi-universal
unfolding of $N_{a_0}^{\left(\lambda\right)}$: actually, after one blowing-up and in the multiform case, the equation (\ref{triv_loc}) becomes  
$$\left\{ \begin{array}{l}
\widetilde{\Omega}\left(X_{kl}^{\left(\epsilon\right)}\right)=0\\
X_{kl}^{\left(\epsilon\right)}=\left(\cdots\right)\frac{\partial}{\partial x_{\epsilon}}+\left(\cdots\right)\frac{\partial}{\partial y_{\epsilon}}+\frac{\partial}{\partial a_{k,l}}\end{array}\right.\epsilon=1,2$$
The same computation as before ensures that the family $\left\{X^{(1)}_{kl}-X^{(2)}_{kl}\right\}_{k,l}$ is a basis of the space of infinitesimal unfolding. 

\noindent Once the separatrices
are normalized, the logarithmic foliation defined by $f_{1}^{\lambda_{1}}\cdots f_{p}^{\lambda_{p}}$
is given by the $1$-form 
\[uN_{a}^{\left(1\right)}\left(\lambda_{1}\frac{dx}{x}+\lambda_{2}\frac{dy}{y}+\lambda_{3}\frac{d\left(x+y\right)}{x+y}+\sum_{i=1}^{p-3}\lambda_{i+3}\frac{dP_{i}\left(x,y\right)}{P_{i}\left(x,y\right)}\right)+|\lambda| N_{a}^{\left(1\right)}du\]
where $u$ is a unity which can be supposed polynomial in the variables
$\left(x,y\right)$. Decomposing $u$ in homogeneous components yields
the following unfolding \[
\Omega_{t_{1},\ldots,t_{M},a}=\left(1+\sum_{i=1}^{M}t_{i}u_{i}\right)\Omega_{a}+|\lambda|d\left(\sum_{i=1}^{M}t_{i}u_{i}\right)N_{a}^{(1)}\] with parameters in $\mathbb{C}^{M}\times A$.
\noindent This deformation has the same property as the deformation $F_{t,a}$
with respect to the action of $\mathbb{C}^{*}$: if $\beta\in \mathbb{C}^*$ stands for the homothety $(x,y)\mapsto \beta(x,y)$,
\[ \beta^*\Omega_{t_{1},\ldots,t_{M},a}=\beta^p\Omega_{\beta t_{1},\ldots,\beta^M t_{M},\beta\cdot a} \]
Following the argument of (\ref{existence}), this ensures the statement of existence of normal forms. 

\noindent Surprisingly, we can avoid the argument of \cite{Cer-Ber} to prove the unicity
of the normal forms since $\Omega_{a}$ does not admit any uniform
first integral. Actually, suppose that we have a conjugacy relation
between two normal forms $N_a^{(\lambda)}$ and $N_b^{(\lambda)}$. Let $\Omega_a=dN_a^{(\lambda)}/N_a^{(\lambda)}$ and $\Omega_b=dN_b^{(\lambda)}/N_b^{(\lambda)}$. We have\[
\phi^{*}\Omega_{a}=u\Omega_{b}\]
where $u$ is a unity. Since $\Omega_{a}$ and $\Omega_{b}$ are closed,
the following relation holds\[
du\wedge\Omega_{b}=0.\]
Therefore, $u$ must be a uniform first integral and thus, a constant
function. Composing on the right side by an homothety yields \[
\phi^{*}\Omega_{a}=\Omega_{\beta\cdot b}\]
for a $\beta$ in $\mathbb{C}^{*}$. Writing $\phi$ as the flow of formal vector field, the unicity can be proved in much same way as the uniform case in (\ref{unicity}). \end{proof}

\section{The moduli space of curves}

This section is devoted to the study of the distribution $\mathcal{C}$ related to the following equivalence relation on the moduli space or $M^{(\lambda)}$ of foliations: two points in $M^{(\lambda)}$ are equivalent if and only if the separatrices of the corresponding class of foliations are in the same analytic class of curves. Therefore, the moduli space of curves in the topological class $\mathcal{T}^{(\lambda)}$ for the Darboux case, is the quotient space of this distribution. Note that the following description of $\mathcal{C}$ works for any integral or complex values of the multiplicities $(\lambda)$.

\medskip
\noindent We begin with the local case. We fix a parameter $a$ in $A$. Let $\FF$ be the corresponding foliation defined by the normal form $N_{a}^{(n)}$, $S$ its separatrix, $\Theta_\FF$ the sheaf restricted on the divisor $D$ of the germs of vector fields tangent to the desingularized foliation $\widetilde{\FF}$, and $\Theta_{S}$ the sheaf of germs of vector fields tangent to $\widetilde{S}$. Recall that the tangent space to the local moduli space of foliations at this point is $H^1(D,\Theta_\FF)$.
% and that the vector fields $\partial/\partial a_{k,l}$ induce a basis of this vector field through the %process described in \S 2. 
According to \cite{Mattei1}, the subspace $H^1(D,\Theta_{S})$ is the tangent space to the moduli space of curves at $[S]$, and we have the following exact sequence of sheaves 
$$0\rightarrow \Theta_\FF \rightarrow \Theta_{S} \rightarrow \mathcal{O}(S)\rightarrow 0$$
where $\mathcal{O}(S)$ is the sheaf of ideals generated by the pullback of the reduced equation $(N_{a}^{(1)}\circ E)$ of $S$. The first morphism is the trivial inclusion, and the second one is induced by the evalution by the holomorphic one-form $\widetilde{\omega}=E^*N_{a}^{(1)}dN_{a}^{(n)}/N_{a}^{(n)}$. We have:
\begin{eqnarray*}
&&H^0(D,\mathcal{O}(S))= \mathcal{O}\widetilde{N}_{a}^{(1)}=\{hN_{a}^{(1)}\circ E,\ h\in\mathcal{O}\}.\\
&&H^0(D,\Theta_{S})=\{X,\ X(N_{a}^{(1)})\in (N_{a}^{(1)})\}.
\end{eqnarray*}
Note that the second module is the module of logarithmic vector fields, i.e. the vector fields at $(\C^2,0)$ tangent to $S$, usually denoted by $\mathcal{X}(\log S)$. From the long exact sequence in cohomology we deduce the following exact sequence
 
$$0\rightarrow \mathcal{O}\cdot N_{a}^{(1)}\left/\widetilde{\omega}\left(\mathcal{X}(\log S)\right)\right. \xrightarrow{\delta} H^1(D,\Theta_\FF)\rightarrow H^1(D,\Theta_{S})\rightarrow 0.$$
Hence, we have the following lemma

\begin{lemma} The image of $\mathcal{O}\cdot N_{a}^{(1)}\left/\widetilde{\omega}\left(\mathcal{X}(\log S)\right)\right.$ by $\delta$ in $H^1(D,\Theta_\FF)$ defines a distribution of vector fields on $A$ which corresponds to the infinitesimal deformations which let invariant the separatrix of the foliation. 
\end{lemma}

In order to study this distribution we introduce the following definition:

\begin{defi} Let $f=f_1^{n_1}\cdots f_p^{n_p}$ be a germ of function in the topological class $\mathcal{T}^{(n)}$, and let $\Omega$ be the closed logarithmic one-form $\sum_i n_idf_i/f_i$. Let $\mathcal M$, $D$, $\widetilde{\Omega}$ be the manifold, divisor and one-form obtained by blowing-up.  An open set $U$ in $\mathcal M$ is a \textit{quasi-homogeneous open set }with respect to $f$ if there exists a vector field $R_U$ on $U$ such that
$\widetilde{\Omega}(R_U)=1$, i.e. $R_U(\widetilde{f})=\widetilde{f}$.
\end{defi}

\begin{rema}\label{remarques_qh}
\begin{enumerate}
	\item Such a vector field $R_U$ is a \textit{local transverse symmetry} for the logarithmic form $\Omega$  and we have the existence of a global symmetry for $\Omega$ if and only if $f$ is quasi-homogeneous (see \cite{Cerveau-Mattei}), what justifies the terminology. 
	%\item The existence of a symmetry only depends on the foliation defined by $f$. 
	\item Two symmetries on $U$ with respect to $\Omega$ differ from a vector field tangent to the foliation.
	\item The two domains $U_1$ and $U_2$ of the charts on $\mathcal M$ are quasi-homogeneous for $N_{a}^{(n)}$: from the local expression of $N_{a}^{(n)}\circ E$ one can easily check that we can find symmetries $R_1$ and $R_2$ on each open set. Therefore, the cocycle $\{R_1-R_2\}$ defines an element of $H^1(D,\Theta_\FF)$ denoted by $[R_1-R_2]$.
	\item This definition, and the previous remarks can be extended in the Darboux case since everything can be written by using $\Omega=\sum_i \lambda_idf_i/f_i$.
\end{enumerate}
\end{rema}

\begin{theorem}\label{qh_engendre}
\begin{enumerate}
\item The set of cocycles $h\cdot[R_1-R_2]:=[\widetilde{h}(R_1-R_2)]$, $h\in \mathcal{O}_2$, defines the distribution $\mathcal{C}$ of foliations with fixed separatrices.
\item This distribution $\mathcal{C}$ is a singular integrable one.
\item The finite family of vector fields $X_{i,j}:=x^iy^j\cdot [R_1-R_2]$, $i+j\leq p-4$ generates the distribution $\mathcal{C}$.
\end{enumerate}
\end{theorem}

\begin{proof}
$(1)-$ We recall the definition of $\delta$: for any element $hN_a^{(1)}\circ E$ of $H^0(D, \mathcal{O}(S))$, one can locally solve the equation
$$\widetilde{\omega}(X_U)=hN_{a}^{(1)}\circ E.$$
Indeed, on each $U_i$, $i=1,2$, since $\widetilde{\omega}=E^*N_{a}^{(1)}dN_{a}^{(n)}/N_{a}^{(n)}$, this equation is equivalent to 
$d\widetilde{N}_{a}^{(n)}(X_U)=\widetilde{N}_{a}^{(n)}$, and
$X_U=hR_i$ satisfies this equation. Now, by definition of $\delta$, the cocycle $\widetilde{h}[R_1-R_2]$ is the image of $hN_a^{(1)}$ in $H^1(D,\Theta_\FF)$.\medskip

\noindent $(2)-$ To check the integrability of this distribution, we consider
the natural projection $\pi$ from $(\mathbb{C}^2,0)\times A$ onto $A$. Let $X$ be any germ of vector field around $a^0$ in $A$. The deformation of the foliation $N_a^{(n)},\ a\in (A,a^0)$ along the trajectories of $X$ lets invariant the analytical class of the curve $N_a=0$ if and only if there exists a germ of vector field $\widetilde{X}$ in $(\mathbb{C}^2,0)\times (A,a^0)$ tangent to the hypersurface $ N_a(x,y)=0$  in $(\mathbb{C}^2,0)\times (A,a^0)$ such that
$\dd \pi (\widetilde{X})=X.$ Since $\dd \pi $ commutes with the Lie bracket, the distribution is involutive and therefore, integrable. 

\noindent Clearly this distribution is singular: for example, at a point $\overline{a^0}$ --the zero matrix except on the first line-- corresponding to an homogeneous foliation, according to the first point of remark (\ref{remarques_qh}), the distribution is reduced to $\{0\}$, since the cocycle $[R_1-R_2]$ is here trivial.\medskip

\noindent $(3)-$ This is a consequence of the first point and of the following remark: if $\nu_0(h)\geq p-3$ then the cocycle $h\cdot[R_1-R_2]$ is trivial. This fact can be easily checked from the description of $H^1(D,\Theta_\FF)$ in \S 2, using the basis $[x_1^{k-1}y_1^{-l}]$ obtained by J.F. Mattei in \cite{Mattei2}.
\end{proof}

\noindent Note that the arguments involved in the proof of the two firts points of this theorem are sufficiently general to be extended in other non generic topological classes. 

Now, we give the expression of $X_{0,0}$ in the basis $\fraction{\partial  }{\partial a_{k,l}}$ of $A$:

\begin{prop}\label{cocycleQH}
$$X_{0,0}=\left[R_1- R_2\right]=\fraction{1}{|n|}\sum_{l\geq 1,k\leq l}(k-1)a_{k,l} \fraction{\partial  }{\partial a_{k,l}}.$$
\end{prop}

\begin{proof}
Let $a_0$ be in $A$ and consider the following deformation 
$$\left(\lambda , a)\in (\mathbb{C},1)\times(A,a_0)\right)\mapsto N_{a,\lambda}(x,y)=N_a^{(n)}(\lambda x,\lambda y)=\lambda^{|n|} N_{\lambda\cdot a}^{(n)}(x,y).$$
This deformation is analytically trivial in $\lambda$. Hence, its related cocycle is trivial. Now, blowing-up the deformation yields
\begin{eqnarray*}
\widetilde{N}_{a,\lambda}(x_1,y_1)&=&\lambda^{|n|} x_1^{|n|} y_1^{n_2}(y_1+1)^{n_3} \prod_{l=1}^{p-3}(y_1+\sum_{k\leq l}\lambda^{k-1}a_{k,l}x_1^{k-1})^{n_{l+3}}\\
\widetilde{N}_{a,\lambda}^{-1} \frac{\partial\widetilde{N}_{a,\lambda}}{\partial \lambda}&=&
{|n|}\lambda^{-1}+\sum_l n_{l+3}\frac{\sum_{k\leq l}(k-1)\lambda^{k-2}a_{k,l}x_1^{k-1}}{y_1+\sum_{k\leq l}\lambda^{k-1}a_{k,l}x_1^{k-1}}\\
&=&{|n|}\lambda^{-1}+\sum_{l,k\leq l}(k-1)\lambda^{k-2}a_{k,l}\widetilde{N}_{a,\lambda}^{-1}\fraction{\partial \widetilde{N}_{a,\lambda} }{ \partial a_{k,l} }.
\end{eqnarray*}
The vector field $R_1$ is defined as a solution on $U_1$ of $R_1 \widetilde{N}_a=\widetilde{N}_a$. Moreover, $\frac{\partial}{\partial a_{k,l}}$ is defined by the cocycle related to the vector fields $X^{(i)}_{k,l}$ such that
$X^{(i)}_{k,l}\widetilde{N}_a=\fraction{\partial}{\partial a_{k,l}} \widetilde{N}_{ a}$.
Setting $\lambda=1$, we obtain
\begin{eqnarray*}
\fraction{\partial \widetilde{N}_{a,\lambda} }{\partial \lambda} |_{\lambda=1}&=&\left({|n|}R_1+\sum_{l,k\leq l}(k-1)a_{k,l}X^{(1)}_{k,l}\right)\widetilde{N}_{a,\lambda}\\
&=&Y^{(1)}\widetilde{N}_{a,\lambda}.
\end{eqnarray*}
In the second chart of the blowing up, the same computation leads to
$$Y^{(2)}={|n|}R_2+\sum_{l,k\leq l}(k-1)a_{k,l}X^{(2)}_{k,l}.$$
Since the cocycle $Y^{(1)}-Y^{(2)}$ is trivial, we obtain
$${|n|}\left[R_1-R_2\right] -\sum_{k\geq 1,l} (k-1)a_{k,l}\left[X^{(1)}_{k,l}-X^{(2)}_{k,l}\right]=0.$$
\end{proof}
\noindent We are not able to give similar general expressions for all the others generators $X_{i,j}$ of $\mathcal C$, obtained by the action of $x\cdot$ and $y\cdot$ on $X_{0,0}$. Nevertheless, we give now some informations about these vector fields. 
%We set $a_k:=\{a_{k,l}, l=k,\cdots p-3\}$. For each vector field $X$ on $A$ we denote $X^k$ its projection %on the ''k-level'' of $A$ defined by its components on $\partial/\partial a_{kl}$, $l=k\cdots p-3$.

%\begin{lemma}
%$X_{k,l}$ is polynomial in the variables $a_{m\cdot}$ with $m\geq 2$ and rational in $a_{1\cdot}$
%\end{lemma}
%\begin{proof}
%The coefficient of $X_{0,0}$ in the basis $\left[\fraction{x^i}{y^j}\right]$ are rational in the variables $a$. Since the action of $x\cdot$ and $y\cdot$ is just a shift in the basis $\left[\fraction{x^i}{y^j}\right]$, the coefficient of $X_{k,l}$ are also rational. So are the coefficient of the change of basis from $\left[\fraction{x^i}{y^j}\right]$ to $\fraction{\partial}{\partial a_{kl}}$. Hence the coefficient of $X_{k,l}$ in the basis $\fraction{\partial}{\partial a_{kl}}$ are rational. But once the coefficient $a_{1\cdot}$ are fixe, $X_{k,l}$ must be definite on the whole set of parameter. Hence, the coefficient are actually polynomial.
%\end{proof}

\begin{prop}\label{XZ}
The vector field $X_{i,j}$ on $A$ is related to the cocycle $x^iy^j\cdot [R_U-R_V]$ if and only if there exists a germ of vector field 
$$Z_{i,j}=\alpha_{i,j}(x,t,a)\fraction{\partial}{\partial x}+\beta_{i,j}(x,t,a)\fraction{\partial}{\partial y}$$ such that 
\begin{equation}\label{compute}
X_{i,j}\cdot N^{(n)}_a=Z_{i,j}\cdot N^{(n)}_a+x^iy^jN^{(n)}_a.
\end{equation}
%Furthermore, the vector field $Z_{i,j}$ satisfying (\ref{compute}) is unique up to a tangent vector field.
\end{prop}

\begin{proof}
The vector field $X_{i,j}$ as derivation on $A$ is the image of the cocycle $[X_{i,j}^U-X_{i,j}^V]$ in $H^1(D,\Theta_\FF)$ if and only if 
\begin{eqnarray*}
X_{i,j}({N}_a^{(n)}\circ E)&=&X_{i,j}^U({N}_a^{(n)}\circ E)\\
&=&X_{i,j}^V({N}_a^{(n)}\circ E).
\end{eqnarray*}
Now the equality $[X_{i,j}^U-X_{i,j}^V]=x^iy^j\cdot [R_U-R_V]$ in $H^1(D,\Theta_\FF)$ means that there exist tangent vector fields $T_{i,j}^U$ and $T_{i,j}^V$ such that
$$X_{i,j}^U-X_{i,j}^V=T_{i,j}^U+(x^iy^j\circ E)R_U-(x^iy^j\circ E)R_V-T_{i,j}^V.$$
Therefore the local vector fields $X_{i,j}^U-(x^iy^j\circ E)R_U-T_{i,j}^U$ and $X_{i,j}^V-(x^iy^j\circ E)R_V-T_{i,j}^V$ glue together in a global vector field $\widetilde{Z}_{i,j}$ which satisfy the equation obtained by the pull back by $E$ of the equation (\ref{compute}). By blowing down, we obtain a vector field ${Z}_{i,j}$ which satisfy (\ref{compute}). 

\noindent To the converse, we fix $h=x^iy^j$, and we consider a pair $(X,Z)$, $X$ derivation on $(\C^2,0)$ whose components depend on $x$, $y$ and $a$, $Z$ derivation on $A$, which satisfy (\ref{compute}). By difference, we obtain:
$$(X_{i,j}-X)\cdot N^{(n)}_a=(Z_{i,j}-Z)\cdot N^{(n)}_a.$$
Therefore the cocycle obtained by pullback of $Z_{i,j}-Z$ and restriction to $U$ and $V$ is a trivial cocycle related to the derivation $(X_{i,j}-X)$. Since $A$ is the base space of a semi-universal deformation, $X=X_{i,j}$.
\end{proof}
%\noindent 
%The lemma is simply a practical way to assert that the class $x^ky^l\left[R_U-R_V\right]$ is cohomologue to $X_{kl}$ and, since $R_U\cdot \tilde{N}_a=\tilde{N}_a$, its proof is the traduction in cohomological language of this statement.  Notice that such a relation is unique: suppose that there exists $X^{'}_{k,l}$ a vector field in the space of parameters $A$ and a vector field $Z^{'}_{k,l}$ such that 
%$$x^ky^lN^{(1)}_a+Z^{'}_{k,l}\cdot N^{(1)}_a=X^{'}_{k,l}\cdot N^{(1)}_a$$
%then 
%$$\left(Z^{'}_{k,l}-Z_{k,l}\right)\cdot N^{(1)}_a=\left(X^{'}_{k,l}-X_{k,l}\right)\cdot N^{(1)}_a.$$
%This relation ensures the triviality of the cocycle $ X^{'}_{k,l}-X_{k,l}$ and thus the equality $X^{'}_{k,l}=X_{k,l}$ since the space of parameters $A$ is the base space of a semi-universal deformation.

\begin{rema}

\begin{enumerate}
	\item Note that the pair $(X_{0,0},Z_{0,0})$, $X_{0,0}$ defined in Proposition (\ref{cocycleQH}), $Z_{0,0}:=-(x\partial/\partial x+y\partial/\partial y)$ satisfies (\ref{compute}) for $i=j=0$, which gives another proof of this proposition. 
	\item Proposition (\ref{XZ}) appears to be the nice way to compute the vector fields $X_{i,j}$ by formal iterative calculation of the pair $(X_{i,j}, Z_{i,j})$ (see an example at the end of this section).
	\item The argument in the proof of proposition (\ref{XZ}) is sufficiently general to be generalized in non generic topological classes.
\end{enumerate}
\end{rema}

\noindent\textbf{Notations.} Let $A=\oplus_{m=1}^{p-3}A^m$ the direct decomposition of $A$, where the level $A^m$ of height $m$ is the $p-m+2$ dimensional vector space generated by the vector fields $\partial/\partial a_{m,l}$, $l=m,\cdots p-3$.
The decomposition of each vector field $X$ on $\oplus_{m=1}^{p-3}A^m$ is denoted by
$$X=X^{\nu}+X^{\nu+1}+\cdots +X^{p-3}$$
where $X^{\nu}$ is the first non vanishing component of $X$.
For simplicity, we shall denote $a_m$ for $\{a_{m,l},\ l=m,\cdots p-3\}.$

%
%\begin{prop}\label{crochets} Let $X$ be the lie bracket $\left[X_{i,j},X_{k,l}\right]$ with $i+j=m-2$ and $k+l=n-2$. If $m\neq n$ then $X$ has its component of smallest height at height $m+n-2$. Moreover, if $m=n$ then it has no component of height smaller than $2m-1$. 
%\end{prop}
%
%\begin{proof}
%The lemma relies on the following calculus
%\begin{eqnarray*}
%\left[X_{k,l},X_{i,j}\right]\cdot N_{a}^{\left(1\right)} & = & \left(Z_{i,j}\left(x^{k}y^{l}\right)-Z_{k,l}\left(x^{i}y^{j}\right)\right)N_{a}^{\left(1\right)}\\
% &  & +\left[X_{k,l},Z_{i,j}\right]\cdot N_{a}^{\left(1\right)}-\left[X_{i,j},Z_{m}\right]\cdot N_{a}^{\left(1\right)}+\left[Z_{i,j},Z_{k,l}\right]\cdot N_{a}^{\left(1\right)}\end{eqnarray*} Looking at the homogeneous part of lowest degree in (\ref{compute}), it appears that $Z_{k,l}$ has a valuation equal to $k+l+1$ and that  
%$$Z=x^ky^l\left(x\fraction{\partial}{\partial x}+y\fraction{\partial}{\partial y}\right)+\cdots.$$
%Hence,
%$$Z_{i,j}\left(x^{k}y^{l}\right)-Z_{k,l}\left(x^{i}y^{j}\right)=\left(i+j-(k+l)\right)x^{k+i}y^{j+l}+\cdots$$ which proves the lemma. 
%\end{proof}

\begin{lemma}
The coefficients of $X_{k,l}$ in the basis $\frac{\partial}{\partial a_{k,l}}$
are in the ring $\mathbb{C}\left(a_{1}\right)\left[a_{2},\ldots,a_{p-3}\right]$.\end{lemma}

\begin{proof}
Combining the formula (\ref{cocycleQH}) and the proposition (\ref{polynomial}) ensures that the coefficients
of $X_{0,0}$ in the standard basis are in $\mathbb{C}\left(a_{1}\right)\left[a_{2},\ldots,a_{p-3}\right]$.
Since the multiplication by $x^{k}y^{l}$ is a linear shift, the coefficients
of $X_{k,l}$ are also in the previous ring. Now, if we order the basis
$\left\{ \left[\frac{x_{1}^{i}}{y_{1}^{j}}\right]\right\} _{i,j}$
and $\left\{ \frac{\partial}{\partial a_{k,l}}\right\} _{k,l}$
using the lexicographic order on $\mathbb{N}^{2}$ then the matrix
of basis changing is diagonal by blocks with coefficients in $\mathbb{C}\left(a_{1}\right)\left[a_{2},\ldots,a_{p-3}\right]$.
Moreover, the diagonal blocks only depend on the variables $a_{1}$.
Thus the coefficients of the inverse matrix are in $\mathbb{C}\left(a_{1}\right)\left[a_{2},\ldots,a_{p-3}\right]$,
which proves the lemma. 
\end{proof}

\noindent In the next result, we give a description of the vector fields $X_{k,l}$.
\begin{prop}\label{homogeneite}
For any $k,\ l$, we have
$$\left[X_{k,l},X_{0,0}\right]=\left(k+l\right)X_{k,l}.$$ 
The coefficients of $X_{k,l}^\nu$ are homogeneous with respect to the weight $X_{0,0}$ of degree $\nu-(k+l)-1$. In particular, they only depend on the variables $a_{m}$ with $m\leq k-1$. Hence, we get the following expression     
\begin{eqnarray*}
X_{k,l}&=&X_{k,l}^{\nu}+X_{k,l}^{\nu+1}+\cdots X_{k,l}^{p-3}\\
	&=& \sum_{i=\nu}^{p-3}\alpha_{k,l}^{\nu,i}(a_{1},a_{2})\frac{\partial}{\partial a_{\nu,i}}+
	\sum_{i=\nu+1}^{p-3}\alpha_{k,l}^{\nu+1,i}(a_{1},a_{2},a_{3})\frac{\partial}{\partial a_{\nu+1,i}}+\cdots \\
	&\cdots& +\alpha_{k,l}^{p-3,p-3}(a_{1},a_{2},\cdots a_{p-k-1})\frac{\partial}{\partial a_{p-3,p-3}}
\end{eqnarray*}
where $k+l=\nu-2$ and the coefficients $\alpha_{k,l}^{\nu,i}$ are rational in the variables $a_{1}$ and polynomial in the others.
\end{prop}

\begin{proof} For the sake of simplicity, $N$ stands here for $N_a^{(n)}$.
Let $Z$ be vector field $Z=A\left(\cdot\right)\frac{\partial}{\partial x}+B\left(\cdot\right)\frac{\partial}{\partial y}$
such that \[
x^{k}y^{l}N+Z\cdot N=X_{k,l}\cdot N.\]
Let $R$ be the radial vector field. We have verified that $N+R\cdot N=X_{0,0}\cdot N$. Note that $\left[R,X_{k,l}\right]=0$.
Therefore we can perform the following computation  
\begin{eqnarray*}
\left[X_{0,0},X_{k,l}\right]\cdot N & = & X_{0,0}\left(x^{k}y^{l}N+Z\cdot N\right)-X_{k,l}\left(N+R\cdot N\right)\\
 & = & x^{k}y^{l}\left(N+R\cdot N\right)+X_{0,0}\cdot Z\cdot N-x^{k}y^{l}N\\
 &&-Z\cdot N-R\left(x^{k}y^{l}N+Z\cdot N\right)\\
 & = & x^{k}y^{l}R\cdot N+\left[X_{0,0},Z\right]\cdot N+Z\cdot X_{00}\cdot N\\
 &&-Z\cdot N-x^{k}y^{l}R\cdot N-\left(k+l\right)x^{k}y^{l}N-R\cdot Z\cdot N\\
 & = & \left[X_{0,0},Z\right]\cdot N+Z\cdot\left(N+R\cdot N\right)-Z\cdot N\\
 &&-\left(k+l\right)x^{k}y^{l}N-R\cdot Z\cdot N\\
 & = & \left[X_{0,0},Z\right]\cdot N+\left[Z,R\right]\cdot N-\left(k+l\right)x^{k}y^{l}N\end{eqnarray*}Since the vector $\left[X_{0,0},Z\right]$ is written $\left(\cdot\right)\frac{\partial}{\partial x}+\left(\cdot\right)\frac{\partial}{\partial y}$,
the previous relation ensures that $\left[X_{0,0},X_{k,l}\right]=-\left(k+l\right)X_{k,l}$.
Now, if we decompose the vector field $X_{k,l}$ on the basis $\left\{\fraction{\partial}{\partial a_{i,j}}\right\}$ and inject this decomposition in the Lie bracket, it follows
\begin{eqnarray*}
\left[X_{0,0},X_{k,l}\right] & = & \sum_{\nu\geq k+l+2}\left[X_{0,0},X_{k,l}^{\nu}\right]\\
 & = & \sum_{\nu\geq k+l+2}\left[X_{0,0},\sum_{i}\alpha_{k,l}^{\nu,i}\left(a\right)\frac{\partial}{\partial a_{\nu,i}}\right]\\
& = & \sum_{\nu\geq k+l+2}\sum_{i}\left(X_{0,0}\cdot\alpha_{k,l}^{\nu,i}\left(a\right)-\left(\nu-1\right)\alpha_{k,l}^{\nu,i}\left(a\right)\right)\frac{\partial}{\partial a_{\nu,i}}\end{eqnarray*}
Hence, identifying the coefficients in the basis $\left\{\fraction{\partial}{\partial a_{k,l}}\right\}$ leads to the relation 
$$X_{0,0}\cdot\alpha_{k,l}^{\nu,i}\left(a\right)-\left(\nu-1\right)\alpha_{k,l}^{\nu,i}(a)=-(k+l)\alpha_{k,l}^{\nu,i}(a).$$Therefore $\alpha_{k,l}^{\nu,i}$ is homogeneous with respect to the weight $X_{0,0}$. Its weight is $\nu-1-(k+l)$. In particular, since $k+l>1$, $\alpha_{k,l}^{\nu,i}$ does not depend on $a_{j}$ with $j\geq \nu$ since $a_{j}$ is of weight $j-1$. 
\end{proof}

\noindent The next proposition deals with the generic dimension of the distribution $\mathcal{C}$.

\begin{prop}\label{independance}
For any $m=2,\cdots p-3$, the dimension of the vector space generated by the $m-1$ vector fields $X_{k,l}$, $k+l=m-2$ is $$\min\left(m-1,p-m-2\right)$$ 
%For any $0\leq n\leq p-5$ the family $\left\{X_{k,l}\right\}_{k+l=n}$ is generically of dimension $\min\left(n+1,p-4-n\right)$
\end{prop}

\begin{proof} In this proof we can assume, for simplicity, that the $n_i$'s are equal to 1.
Let us write the decomposition of the cocyle $\frac{\partial}{\partial a_{1,l}}$ in the standard basis of $H^1(,D,\Theta_0)$  \[
\frac{\partial}{\partial a_{1,l}}=\sum_{{\scriptsize\begin{array}{c}
1\leq i\leq p-3\\
1\leq j\leq p-2-i\end{array}}}R_{ij}^{l}(a)\left[\frac{x_{1}^{i-1}}{y_{1}^{j}}\right].\]
According to theorem (\ref{thm_local}) and using the notation introduced in its proof, the cocycle associated to $\frac{\partial}{\partial a_{1,l}}$ is written 
$$\fraction{n_{l+3}R(y_1)}{|n|P(y_1)}\left(U(y_1)+\fraction{\tilde{W}(y_1)}{y_1^{2p-5}}\right)+x_1\left(\cdots\right).$$
The polynome $R=\frac{P}{P\wedge P'}$ is equal to $P$. Therefore, the coefficient $R_{1j}^{l}(a)$
is the coefficient of $\frac{1}{y_{1}^{j}}$ in the developement in Laurent series
of $\frac{\tilde{W}\left(y_{1}\right)}{py_{1}^{2p-5}\left(y_{1}+a_{1,l}\right)}$.
Note that it only depends on the variables $a_{1}$. It is clear that $y_1$ does not divide $\tilde{W}$ and that the degree of $\tilde{W}$ is at most $p-2$. Hence, the cocycle is developped in 
$$\fraction{1}{pa_{1,l}y_1^{2p-5}}\underbrace{\sum_{i=0}^{p-2}w_iy_1^i}_{\tilde{W}}\sum_{j=0}^\infty\fraction{y_1^j(-1)^j}{a_{1,l}^j}=\sum_{j=0}^\infty \left(\sum_{i=0}^{\min(j,p-2)}\fraction{w_i(-1)^{j-i}}{pa_{1,l}^{j-i+1}}\right)\fraction{1}{y_1^{2p-5-j}}.$$   
For $j\geq p-2$ the coefficient of $\fraction{1}{y_1^{2p-5-j}}$ is $$\sum_{i=0}^{p-2}\fraction{w_i(-1)^{j-i}}{pa_{1,l}^{j-i+1}}=\fraction{(-1)^j}{pa_{1,l}^j}\sum_{i=0}^{p-2}w_i(-1)^ia_{1,l}^{i-1} $$
Therefore, the functions $R^l_{1j}$ satisfy \[
R_{1j}^{l}\left(a_{1,\cdot}\right)=\left(-1\right)^{j-1}a_{1,l}^{j-1}R_{11}^{l}\left(a_{1,\cdot}\right)\]
Using the lemma $\left(\ref{cocycleQH}\right)$ yields a decomposition of
the cocycle $X_{0,0}$ in the base $\left\{ \left[\frac{x_{1}^{i-1}}{y_{1}^{j}}\right]\right\} _{{\scriptsize\begin{array}{c}
1\leq i\leq p-3\\
1\leq j\leq p-2-i\end{array}}}$ 
whose first terms are \[
X_{0,0}=\frac{1}{n}\sum_{1\leq j\leq p-4}\underbrace{\left(\sum_{1\leq l\leq p-4}a_{2,l}\left(-1\right)^{j-1}a_{1,l}^{j-1}R_{11}^{l}\right)}_{X_{0,0}^j}\left[\frac{x_{1}}{y_{1}^{j}}\right]+\cdots,\]
 where the terms in the dots belong to $\textup{Vect}\left\{ \left[\frac{x_{1}^{i-1}}{y_{1}^{j}}\right]\right\} _{{\scriptsize\begin{array}{c}
1\leq i\leq p-3\\
1\leq j\leq p-2-i\end{array}}}$. Therefore, for any $k,l$ the following relation holds

\[X_{k,l}=\frac{1}{n}\sum_{1\leq j\leq p-4-(k+l)}X_{0,0}^{j+l} \left[\frac{x_{1}^{k+l+1}}{y_{1}^{j}}\right]+\cdots.\]
Now, the $p-4$ functions $X_{0,0}^j$,
$1\leq i\leq p-4$ are algebraically independent. Indeed, consider these functions as linear functions of the $p-4$ variables
$a_{2,l}$, $1\leq l\leq p-4$; their determinant is equal to\[
\textup{Vandermonde}\left(-a_{1,1},\ldots,-a_{1,p-4}\right)\prod_{i=1}^{p-4}R_{1,1}^{i}\]
 which is not the zero function. Hence, for a generic choice of the variables
$a_{1,1},\ldots,a_{1,p-4}$ these $p-4$ linear functions are independent
as linear functions of $p-4$ variables: thus, these are algebraically
independent. Hence, if $k+l=n$ consider the family of vector field $X^{(n+2)}_{k,l}$ which are the projection of $X_{k,l}$ on the $n+2$ level for $1\leq l\leq d=\min(n+1,p-4-n)$. The determinant of this family is  $$
\left|\begin{array}{cccc}
X_{0,0}^{1} & X_{0,0}^{2} & \ldots & X_{0,0}^{d}\\
X_{0,0}^{2} & X_{0,0}^{3} & \ldots & X_{0,0}^{d+1}\\
\\X_{0,0}^{d} &  &  & X_{0,0}^{2d}\end{array}\right|$$
which cannot vanish since it would be a non trivial algebraic relation. The same argument ensures the generic freeness for any sub-family of cardinal $d$ of the family $\left\{X_{k,l}\right\}_{k+l=n}$. 
\end{proof}

%\noindent In the course of the previous proof, we obtained the following corollary which is going to be useful in the next sections:
%
%\begin{coro}
%Any sub family of order less than $\min(n+1,p-4-n)$ of the family $\left\{X^{(n-2)}_{kl}\right\}_{k+l=n}$ is free.
%\end{coro}

\bigskip

From this result, we shall deduce the generic normal forms for curves.
Let $B$ be the linear subset of $A$ defined the equations:
$$\begin{array}{l}
	a_{2,2}=1\\
	a_{3,3}=a_{3,4}=0\\
	a_{4,4}=a_{4,5}=a_{4,6}=0\\
	\cdots\\
	a_{k,l}=0 \mbox{  for } l=k,\cdots \inf\{k-1,p-k-2\}\\
\end{array}$$
For example, the general type of a matrix $b$ in $B$ for $p=10$ branches is given by the $7\times 7$ upper triangular matrix:
\begin{center}
$\begin{array}{ccccccc}
\times & \times & \times & \times & \times & \times  & \times \\
& 1 & \times & \times & \times & \times & \times   \\
&& 0 & 0 & \times & \times & \times \\
&&& 0 & 0 & 0 & \times \\
&&&& 0 & 0 & 0 \\
&&&&&0 & 0 \\
&&&&&& 0 \\
\end{array}$
\end{center}
The dimension of $B$ is $\tau=(p-3)+(p-5)+\cdots +1=\frac{(p-2)^2}{4}$ if $p$ is even, or $(p-3)+(p-5)+\cdots +0=\frac{(p-1)(p-3)}{4}$ if $p$ is odd.

\bigskip

\noindent The following result was previously obtained by M. Granger \cite{Granger} in a completely different way.

%\noindent Now, the $p-4$ distributions $\left\{X_{k,l}\right\}_{k+l=n}$ $n=0\ldots p-5$  are tranverse. Hence by summing the contribution of each height to the total dimension, we recover a result obtained by Garnier in a completely different way:

\begin{theorem}
Each generic leaf of the distribution $\mathcal C$ on A meets $B$ at a unique point $b$ in $B$. Therefore $N_b$, $b\in B$ is the generic normal form for a curve with $p$ branches smooth and transversal, and the generic dimension of the moduli space is $\tau=\frac{(p-2)^2}{4}$ if $p$ is even, or $\tau=\frac{(p-1)(p-3)}{4}$ if $p$ is odd.
\end{theorem}

\begin{proof} We consider a leaf $L_p$ of $\mathcal C$ passing through a point $p$ in $A$. First, since none of the vector fiels $X_{i,j}$ has components on the first level, the values $a_{1,l}$ are invariant along each leaf of $\mathcal C$, and $L_p$ determines a unique point $a_{1,l}(p)$ on this first level (the cross ratios of the tangent cone of the curve).\\
The only generator of $\mathcal C$ which acts on this second level is $X_{0,0}$. Since $X_{0,0}^2$ is the radial vector field, its flow acts by multiplication on $A^2$. Under the generic assumption $a_{2,2}(p)\neq 0$, we can make use of this flow to normalize the value $a_{2,2}$ to 1. Then the others values $a_{2,l}$ are uniquely determined, and correspond to the unique point $q$ intersection of $L_p$ with $a_{2,2}=1$.\\
On the third level there are only two generators of $\mathcal C$ which actually act on this level: $X_{1,0}$ and $X_{0,1}$. According to proposition (\ref{homogeneite}) the components of their initial parts $X_{1,0}^3$ and $X_{0,1}^3$ only depend on $a_1$ and $a_2$ variables. Suppose that $X_{1,0}^3(q)$ and $X_{0,1}^3(q)$ are independent: this is a generic assumption on the leaf $L_p$ since we know from the previous proposition (\ref{independance}) that they are functionally independent. Now, since $X_{1,0}^3$ and $X_{0,1}^3$ are constant vector fields relatively to the coordinates $a_3$, their flows act by translation on the third level, and we make use of it to normalize the two first coordinates of this level to the value 0.\\
We continue on the next levels in the same way. Note that when the number $m-1$ of vector fields acting on $A^m$ becomes greater or equal to the dimension $p-m-2$ of $A^m$, this action becomes generically transitive on $\oplus_{n\geq m}A^n$ and all the coefficients may be normalized to 0.
\end{proof}

The previous theorem only gives a generic description of the quotient space of $\mathcal C$. This one will be obtained by the following result:
	
\begin{theorem}\label{r-integrable}
The distribution $\mathcal{C}$ is rationally integrable: there exists $\tau$ independent rational first integrals for $\mathcal{C}$.
\end{theorem}

\noindent We shall give two proofs of this result. The first one is a consequence of general facts about algebraic actions of algebraic groups. The second one is an algorithmic one, and furthermore proves that we can choose these first integrals in the ring $\C(a_1,a_2)[a_3,\cdots a_{p-3}]$. They define a complete system of invariants for the curves with $p$ smooth and transversal branches.

\bigskip
\textit{First proof of }(\ref{r-integrable}).
Following a result of Mather, let $N$ be an integer of finite determinacy
for the whole topological class of a $N_{a}$. This integer has
the following property: for any function $f$ topologically equivalent
to $N_{a}$, $J^{N}f$ is analytically equivalent to $f$. Since the
property still holds for any bigger integer, we can choose $N$ such
that the application \begin{equation}\label{eq:1}
\begin{cases}
A\longrightarrow J^{N}\mathbb{C}\left\{ x,y\right\} \\
a\longmapsto N_{a}\left(x,y\right)\end{cases}\end{equation}
is an analytical immersion for the evident analytic structures, and an imbedding on a Zariski open set $A^*$ of $A$. %where $A^{*}$ is the \emph{generic} component of $A$,
%those elements of $A$ with no vanishing coefficient. 
Now, let us consider the algebraic group $G:=J^{N}\textup{Diff}\left(\mathbb{C}^{2},0\right)\ltimes J^{N}\mathcal{O}_{2}^{*}$
acting on $J^{N}\mathbb{C}\left\{ x,y\right\} $ in the following way:\[
\left(\phi\left(x,y\right),u\left(x,y\right)\right)\cdot J^{N}f\left(x,y\right)=J^{N}\left(u\left(x,y\right)\times f\left(\phi\left(x,y\right)\right)\right).\]
Two $N$-jet of functions in the topological class of $N_{a}$ are
in the same orbit if and only if their $0$ level curves are analytically
equivalent: indeed, if $J^{N}g=$ $J^{N}\left(u\times f\circ\phi\right)$
then $g$ is analytically equivalent to $u\times f\circ\phi$ for
$N$ is a determinacy integer for both $f$ and $g$. Hence the curve $\left\{ g=0\right\} $ is analytically equivalent to $\left\{ f=0\right\} $.
The partition of the space $J^{N}\mathbb{C}\left\{ x,y\right\} $
by the orbits induces a partition of the image of the embedding (\ref{eq:1}),
and therefore on $A$. It is clear that this partition
coincides with the partition associated to the distribution $\mathcal{C}$.
Now, a result of Rosenlicht --see \cite{Dolgachev}, theorem (6.1)-- ensures that the action of $G$ on $J^{N}\mathbb{C}\left\{ x,y\right\} $ is
completely integrable by rational first integrals. Taking the restriction
of these first integrals along the image of $A$ by (\ref{eq:1}), we obtain the theorem.

\bigskip
\textit{Second proof of }(\ref{r-integrable}).
Clearly, the functions $f_{1,l}=a_{1,l}$ define $\tau_1=p-3$ first integrals of the distribution $\mathcal C$, since none of the vectors $X_{i,j}$ has components on the first level. The meaning of these invariants is clear: they correspond to the cross ratios of each branch relatively to the three first ones.\\
On the second level, we consider the initial part i.e. the first non vanishing component on the lower level of $X_{0,0}$:
$X_{0,0}^2=\sum_{l= 2}^{p-3}a_{2,l}\frac{\partial}{\partial a_{2,l}}$. Clearly, this radial vector field has $\tau_2=p-5$ rational first integrals: $f_{2,l}=\frac{a_{2,l}}{a_{2,2}}$, $l=3,\cdots p-3$. Since these functions do not depend on the following variables, they are first integrals for $X_{0,0}$ and for all the others $X_{i,j}$ for $i+j>0$.

\noindent We now consider the restriction $\mathcal C^{\geq 3}$ of the distribution $\mathcal C$ on the following levels $\mathcal{A}_k$, $k\geq 3$. Clearly, from proposition (\ref{homogeneite}), $\mathcal C^{\geq 3}$ is still an involutive distribution. First integrals $f$ of this distribution have to satisfy the following system:

$$(S)\ \left\{ \begin{array}{rcl}
 (X_{i,j}^3 + X_{i,j}^4 + \cdots + X_{i,j}^{p-3})(f)&=&0, \mbox{ for $(i,j)$ such that $i+j=1$;}\\
   (X_{i,j}^4 + \cdots + X_{i,j}^{p-3})(f)&=&0,\mbox{ for $(i,j)$ such that $i+j=2$;}\\
&\vdots & \\
    (X_{i,j}^{p-4}  + X_{i,j}^{p-3})(f)&=&0,\mbox{ for $(i,j)$ such that $i+j=p-6$;}\\
    (X_{i,j}^{p-3})(f)&=&0, \mbox{ for $(i,j)$ such that $i+j=p-5$;}\\
\end{array}\right. $$

\noindent Recall that, from proposition (\ref{homogeneite}), on each level $k$ the initial parts $X_{i,j}^{k}$ of the $k-1$ vector fields $X_{i,j}$ ($i+j=k-2$) are constant vector fields with respect to the coordinates $a_k$, $k\geq 3$ with rational coefficients in $(a_1, a_2)$. 
From the independence property (\ref{independance}), we can make use of a linear change of coordinates:
$$(a_{k,l}) \rightarrow (b_{k,m},c_{k,n})$$
$m=1,\cdots ,k-1$, $n=1,\cdots ,(p-k-2)-(k-1)=p-2k-1$
such that the $k-1$ initial vector fields $X_{i,j}^{k}$ are colinear to each $\frac{\partial}{\partial b_{k,m}}$. Note that as soon as $p-2k-1<0$ (i.e. $k>[p/2]$) we only have coordinates $b_{k,m}$ : according to (\ref{independance}), the generators of $\mathcal C$ restricted to this second part of $(S)$ act transitively. The total number of coordinates $c_{k,n}$ which appear on the first part of the system ($k\leq [p/2]$) is the generic codimension of $\mathcal C^{\geq 3}$: $\tau_3=\sum_{k=3}^{[p/2]}(p-2k-1)=(p-7)+(p-9)+\cdots$. 

\noindent After performing these changes of variables, the system $(S)$ admits the following expression:
$$(S)\left\{ \begin{array}{rcl}
\alpha_{k,l}\frac{\partial f}{\partial b_{k,l}}&+&(\sum_m\alpha_{k,l}^{k+1,m}\frac{\partial f}{\partial b_{k+1,m}}+
\sum_n\beta_{k,l}^{k+1,n}\frac{\partial f}{\partial c_{k+1,n}})+\cdots \\
\cdots &+&(\sum_m\alpha_{k,l}^{p-3,m}\frac{\partial f}{\partial b_{p-3,m}}+
\sum_n\beta_{k,l}^{p-3,n}\frac{\partial f}{\partial c_{p-3,n}})=0\\
&&\mbox{for }l=1,\cdots k-1, \mbox{ and } k=3,\cdots p-3.
\end{array}\right.$$
The first coefficients $\alpha_{k,l}$ only depend (rationally) on $a_1,a_2$, the following coefficients $\alpha_{k,l}^{k+1,m}$ are rational in $a_1,a_2$, polynomial in $b_3,c_3$ and so on. In this expression we suppose that the coefficients $\beta_{k,l}^{q,n}$ vanish when the corresponding variable $c_{k,l}^{q,n}$ does not exist, i.e. when $q>[p/2]$. Furthermore, according to the block triangular form of $(S)$ we can eliminate the components in $\frac{\partial }{\partial b_{k+1,m}}$,... $\frac{\partial }{\partial b_{p-3,m}}$ by adding combinaisons of the equations of higher levels. Finally we obtain the equivalent system $(S')$:
$$\left\{ \begin{array}{cccccrr}
\alpha_{3,l}\frac{\partial f}{\partial b_{3,l}}&+&
\sum_n\gamma_{3,l}^{4,n}\frac{\partial f}{\partial c_{4,n}}&+&\sum_n\gamma_{3,l}^{5,n}\frac{\partial f}{\partial c_{5,n}}&+\cdots & +\sum_n\gamma_{3,l}^{[p/2],n}\frac{\partial f}{\partial c_{[p/2],n}}=0\\
&&\alpha_{4,l}\frac{\partial f}{\partial b_{4,l}}&+&
\sum_n\gamma_{4,l}^{5,n}\frac{\partial f}{\partial c_{5,n}}&+\cdots &
+\sum_n\gamma_{4,l}^{[p/2],n}\frac{\partial f}{\partial c_{[p/2],n}}=0\\
&&&&\alpha_{5,l}\frac{\partial f}{\partial b_{5,l}}&+\cdots &+ \sum_n\gamma_{5,l}^{[p/2],n}\frac{\partial f}{\partial c_{[p/2],n}}=0\\
&&&&&\ddots & \\
&&&&&&\alpha_{[p/2],l}\frac{\partial f}{\partial b_{[p/2],l}}=0
\end{array}\right.$$
in which the new coefficients $\gamma$ depend on the variables $a_k$ in a similar way as the previous coefficients $\beta$.
The first line is a subsystem of two equations ($l=1,2$), the second line is a subsystem of three equations ($l=1,2,3$) and so one. 

\noindent Note that 
the functions $c_{k,n}$ are rational first integrals for the distribution $\mathcal C^{\geq k}$, but they are not necessary first integrals for the previous levels. In particular the fonctions $f_{3,n}=c_{3,n}$ are first integrals for the whole distribution $\mathcal C^{\geq 3}$. The general idea is now to make rational change of coordinates in order to diagonalize this system. On the next level $k=4$ we set: 
$$f_{4,n}=c_{4,n}+g_{4,n}(b_{3,\cdot})$$
and we search for polynomial functions $g_{4,n}$ such that the functions $f_{4,n}$ are also first integrals for the previous level $k=3$. They have to satisfy the systems
$$(S_{4,n})\ \left\{ \alpha_{3,l}\frac{\partial g_{4,n}}{\partial b_{3,l}}=-\gamma_{3,l}^{4,n}\ \ l=1,2.\right.$$
We claim that these systems $(S_{4,n})$ are involutive. We denote by $$Y_{3,l}=Y_{3,l}^3+Y_{3,l}^4+\cdots +Y_{3,l}^{[p/2]}$$
each vector field appearing on the first line of $(S')$. According to proposition (\ref{homogeneite}), the Lie bracket $[Y_{3,l},Y_{3,l'}]$, $l\neq l'$, is tangent to the distribution $\mathcal C^{\geq 4}$ and therefore vanishes on the first integral $c_{4,n}$ of this distribution. Thus we have
$$[Y_{3,l},Y_{3,l'}](c_{4,n})=Y_{3,l}^3Y_{3,l'}^4-Y_{3,l'}^3Y_{3,l}^4(c_{4,n})=Y_{3,l}^3(\gamma_{3,l'}^{4,n})-Y_{3,l'}^3(\gamma_{3,l}^{4,n})=0$$
which proves that $(S_{4,n})$ is involutive. Now, since the right hand side of $(S_{4,n})$ is polynomial in $b_{3,\cdot}$ the same holds for the primitive $g_{4,n}$ and for $f_{4,n}$. On the level $k=4$ we now introduce the new rational coordinates:
$$(b_{4,m},c_{4,n})\longrightarrow (b_{4,m},f_{4,n}=c_{4,n}+g_{4,n}(b_{3,\cdot})).$$
The action of this change of coordinates on $(S')$ vanishes the components of $Y_{3,l}^4$ on  
$\frac{\partial }{\partial f_{4,n}}$ and we can still vanish the components of $Y_{3,l}^4$ on  
$\frac{\partial }{\partial b_{4,n}}$ by combinaison with the next line. Finally we obtain a new system in which $Y_{3,l}^4=0$, i.e. which is diagonal on the two first columns. Let us check that we can go on with the next level. We set
$$f_{5,n}=c_{5,n}+g_{5,n}(b_{3,\cdot},b_{4,\cdot})$$
and by substituting in $(S')$ we obtain:
$$(S_{5,n})\ \left\{ \begin{array}{ccl}
\alpha_{3,l}\frac{\partial g_{5,n}}{\partial b_{3,l}}&=&-\gamma_{3,l}^{5,n}\ \ l=1,2\\
\alpha_{4,l}\frac{\partial g_{5,n}}{\partial b_{4,l}}&=&-\gamma_{4,l}^{5,n}\ \ l=1,2,3
\end{array}\right.$$
This system is involutive since the vector fields $[Y_{3,l},Y_{3,l'}]$, $[Y_{3,l},Y_{4,l'}]$, $[Y_{4,l},Y_{4,l'}]$ are tangent to $\mathcal C^{\geq 5}$ and thus vanish on $c_{5,n}$, giving the cross derivatives equalities that we need. We can continue the process since we have the following fact: if $Y_{k,l}=Y_{k,l}^k+Y_{k,l}^q+\cdots$ and
$Y_{k',l'}=Y_{k',l'}^{k'}+Y_{k',l'}^{q}+\cdots$ (without any term between $k$ and $q$, $k'$ and $q$) then $[Y_{k,l},Y_{k',l'}]$ is tangent to $\mathcal C^{\geq q}$: indeed the initial vector fields commute and from proposition (\ref{homogeneite}), $[Y_{k,l}^k,Y_{k',l'}^{q}]$ has only components of level great or equal to $q$. Therefore, we obtain $\tau_3$ independent rational first integrals for $\mathcal C^{\geq 3}$. In order to prove that the whole distribution $\mathcal C$ is rationally integrable, we need the following lemma

\begin{lemma} The first integrals $f_{k,n}$ belong to the ring $\C(a_1,a_2)[a_3,\cdots a_{p-3}]$, and are quasi-homogeneous with respect to the quasi-radial vector field $X_{0,0}$.
\end{lemma}

\begin{proof} From proposition (\ref{homogeneite}), we know that the coefficients of the system $(S)$ are in $\C(a_1)[a_2,a_3,\cdots a_{p-3}]$. Note that the only divisions which appear in the previous algorithm involve coefficients in $\C(a_1,a_2)$. It is enough to check that at any step of the previous algorithm, the $X_{0,0}$-homogeneity property is preserved. The details are left to the reader. 
%\noindent The first change of coordinates is a linear change of coordinates in any variables $a_{k,\cdot}$ with $k\geq 3$ whose coefficients are linear function of the variables $a_{2,\cdot}$ with coefficients in $\mathbb{C}(a_{1,\cdot})$. Thus the coefficients $\alpha_{k,n}$ in the systems $\left(S_{k,n}\right)$ are $X_{0,0}$-homogeneous rational functions depending only on $a_{1,\cdot}$ and ${a_{2,\cdot}}$. Since it depends only $a_{1,\cdot}$ and $a_{2,\cdots}$, this change of coordinates does not introduce new variables in the coefficients: more precisely, in the system $(S)$ the function $\alpha_{k,l}^{i,j}$ depends only on the variables $a_{p,\cdot}$ for $p\leq k$ and are $X_{0,0}$-homogeneous. So are the coefficients $\gamma_{k,l}^{i,j}$ in the substitute system $(S')$. Hence, the solution of each system $\left(S_{k,n})\right)$ can be chosen $X_{0,0}$-homogeneous. 
\end{proof}

\noindent Since $f_{k,n}$ is $X_{0,0}$-homogeneous with degree $m$, by dividing by $a_{2,2}^m$ we obtain a function of degree 0, and therefore a first integral of $X_{0,0}$ which is still a first integral for the remainder of the distribution, since the others generators do not have components on the level 2. This concludes the proof of theorem (\ref{r-integrable}).

\bigskip

\noindent \textbf{Examples.} We are going to present two examples of distribution $\mathcal{C}$, namely a generic curve with 
$9$ and 10 branches. To perform the computation, we fix the first jet
of $f$ to be \[
xy\left(y+x\right)\left(y-x\right)\left(y+2x\right)\left(y-2x\right)\left(y+\frac{x}{4}\right)\left(y-\frac{x}{4}\right)\left(y+\frac{x}{3}\right)\]
which means that in the set of parameters $A$, we fix the first line
of entries in the matrices $a$ in $A$. Actually, if we do not fix these
parameters, the computation becomes far too long and the
expressions obtained are irrevelant. With nine branches, the normal forms are written 
{ \[
\begin{array}{l}
xy\left(y+x\right)\left(y-x\right)\left(y+2x+a_{2,2}x^{2}\right)\left(y-2x+a_{2,3}x^{2}+a_{3,3}x^{3}\right)\cdots\\
\times\left(y+\frac{x}{4}+a_{2,4}x^{2}+a_{3,4}x^{3}+a_{4,4}x^{4}\right)\left(y-\frac{x}{4}+a_{2,5}x^{2}+a_{3,5}x^{3}+a_{4,5}x^{4}+a_{5,5}x^{5}\right)\cdots\\
\times\left(y+\frac{x}{3}+a_{2,6}x^{2}+a_{3,6}x^{3}+a_{4,6}x^{4}+a_{5,6}x^{5}+a_{6,6}x^{6}\right).\end{array}\]
}In this situation the set of matricies of $A$ with fixed first
line of entries is of dimension $15$ and the generic dimension of
$\mathcal{C}$ is $9$. Thus, there are $6$ first integrals. The
distribution $\mathcal{C}$ is generated by \[
X_{0,0},\ X_{1,0},\ X_{0,1},\left\{ \frac{\partial}{\partial a_{kl}}\right\} _{k\geq4}\]
because $\mathcal{C}$ acts transitively beyond the level four. Hence,
we can replace the three first generators of $\mathcal{C}$ by their
projections on the three first levels. Hereafter, we give the expression
of these projections
\begin{eqnarray*}
X_{0,0} & = & \frac{1}{9}\sum_{k=2}^{3}\sum_{l=k}^{7}\left(k-1\right)a_{k,l}\frac{\partial}{\partial a_{k,l}}\\
X_{1,0} & = & -\frac{1}{9}\left(a_{2,2}+a_{2,3}\right)\frac{\partial}{\partial a_{3,3}}+\left(\frac{5}{72}a_{2,2}-\frac{1}{9}a_{2,4}\right)\frac{\partial}{\partial a_{3,4}}\\
 & - & \left(\frac{5}{72}a_{2,2}+\frac{1}{9}a_{2,5}\right)\frac{\partial}{\partial a_{3,5}}+\left(\frac{64}{729}a_{2,2}-\frac{1}{9}a_{2,6}\right)\frac{\partial}{\partial a_{3,6}}\\
X_{0,1} & = & \frac{1}{18}\left(a_{2,2}-a_{2,3}\right)\frac{\partial}{\partial a_{3,3}}+\left(-\frac{5}{144}a_{2,2}+\frac{1}{36}a_{2,4}\right)\frac{\partial}{\partial a_{3,4}}\\
 & + & \left(\frac{5}{144}a_{2,2}-\frac{1}{36}a_{2,5}\right)\frac{\partial}{\partial a_{3,5}}+\left(-\frac{32}{729}a_{2,2}+\frac{1}{27}a_{2,6}\right)\frac{\partial}{\partial a_{3,6}}.\end{eqnarray*}
It is clear that $\left[X_{0,1},X_{1,0}\right]=0$. Moreover, we already
know that $\left[X_{0,0},X_{1,0}\right]=X_{1,0}$ and $\left[X_{0,0},X_{0,1}\right]=X_{0,1}$.
Thus, the distribution is involutive and completely integrable. The first four first integrals
are the four quotients \[
f_{1}=\frac{a_{2,3}}{a_{2,2}},\quad f_{2}=\frac{a_{2,3}}{a_{2,2}},\quad f_{3}=\frac{a_{2,3}}{a_{2,2}},\quad f_{4}=\frac{a_{2,3}}{a_{2,2}}.\]
The two others are linear functions in the $a_{3}$ variables
whose coefficients are rational functions in the $a_{2}$ variables, $X_{0,0}$-homogeneous of degree $-2$. \begin{eqnarray*}
f_{5} & = & -\frac{1}{324}\frac{\left(270a_{2,2}a_{2,6}+216a_{2,4}a_{2,6}-512a_{2,2}a_{2,4}\right)}{a_{2,2}^{4}}a_{3,3}\\
 &  & +\frac{1}{81}\frac{\left(540a_{2,3}a_{2,6}-108a_{2,2}a_{2,6}-512a_{2,2}a_{2,3}\right)}{a_{2,2}^{4}}a_{3,4}\\
 &  & +\frac{\left(2a_{2,2}a_{2,4}+5a_{2,2}a_{2,3}-6a_{2,3}a_{4,2}\right)}{a_{2,2}^{4}}a_{3,6}\\
f_{6} & = & -\frac{1}{324}\frac{\left(1215a_{2,2}a_{2,5}+405a_{2,4}a_{2,2}-1296a_{2,4}a_{2,5}\right)}{a_{2,2}^{4}}a_{3,3}\\
 &  & +\frac{1}{81}\frac{\left(-486a_{2,2}a_{2,5}-405a_{2,2}a_{2,3}-162a_{2,3}a_{2,5}\right)}{a_{2,2}^{4}}a_{3,4}\\
 &  & +\frac{\left(2a_{2,2}a_{2,4}+5a_{2,2}a_{2,3}-6a_{2,3}a_{4,2}\right)}{a_{2,2}^{4}}a_{3,5}\end{eqnarray*}
In the case of ten branches the expression are much more complicated.
For example the projection of $X_{1,0}$ on the four first levels is
written 

{\scriptsize \begin{eqnarray*}
X_{1,0} & = & -\frac{\left(a_{2,3}+a_{2,2}\right)}{10}\frac{\partial}{\partial a_{3,3}}-\left(\frac{a_{2,4}}{10}+\frac{a_{2,2}}{256}\right)\frac{\partial}{\partial a_{3,4}}-\left(\frac{a_{2,5}}{10}-\frac{a_{2,2}}{256}\right)\frac{\partial}{\partial a_{3,5}}\\
 & - & \left(\frac{a_{2,6}}{10}+\frac{2a_{2,2}}{405}\right)\frac{\partial}{\partial a_{3,6}}-\left(\frac{a_{2,7}}{10}-\frac{2a_{2,2}}{405}\right)\frac{\partial}{\partial a_{3,5}}\\
 & + & \left(-\frac{74263}{17694720}a_{2,2}a_{2,3}+\frac{7}{2048}a_{3,3}+\frac{10609}{1966080}a_{2,2}^{2}-\frac{1}{5}a_{3,4}-\frac{179}{17280}a_{2,2}a_{2,4}\right)\frac{\partial}{\partial a_{4,4}}\\
 & + & \left(-\frac{74263}{17694720}a_{2,2}^{2}-\frac{9}{2048}a_{3,3}+\frac{10609}{1966080}a_{2,2}a_{2,3}-\frac{1}{5}a_{3,5}-\frac{179}{17280}a_{2,2}a_{2,5}\right)\frac{\partial}{\partial a_{4,5}}\\
 & + & \left(\frac{74263}{10497600}a_{2,2}^{2}+\frac{1}{243}a_{3,3}-\frac{10609}{2099520}a_{2,2}a_{2,3}-\frac{1}{5}a_{3,6}-\frac{311}{43200}a_{2,2}a_{2,6}\right)\frac{\partial}{\partial a_{4,6}}\\
 & + & \left(\frac{74263}{10497600}a_{2,2}a_{2,3}-\frac{7}{1215}a_{3,3}-\frac{10609}{2099520}a_{2,2}^{2}-\frac{1}{5}a_{3,7}-\frac{311}{43200}a_{2,2}a_{2,7}\right)\frac{\partial}{\partial a_{4,4}}\end{eqnarray*}
}This is the smallest case with revelant quadratic terms appearing
in the expressions. There are $9$ first integrals outside the 7 fixed cross-ratios. The eight first
are quite easy to compute even by hand but the ninth appears to be
a rational function whose numerator is an $X_{0,0}$-homogeneous
polynomial function of degree $8$ in $11$ variables with more than
$200$ monomial terms.

\bigskip 
\noindent
\textbf{The global moduli space of curves.}
We have obtained the global moduli space $M^{(n)}$ for the foliations with mutiplicities $(n)$ by quotienting the local moduli space $A$ under the action of $\C^*$: $\lambda\cdot (a_{k,l})=(\lambda^{k-1}a_{k,l}).$ It turns out that this action is exactly the flow at the time $t$ such that $\lambda=e^t$ of the quasi-radial vector field $X_{0,0}$, which is the first generator of the distribution $\mathcal C$ on $A$. Therefore this distribution induces a distribution $\mathcal C'$ on $M^{(n)}$. In order to obtain vector fields which generate $\mathcal C'$, one can remark that from the relation
$$\left[X_{k,l},X_{0,0}\right]=\left(k+l\right)X_{k,l}$$ 
obtained in proposition (\ref{homogeneite}), we deduce that the vector fields 
$a_{2,2}^{k+l}X_{k,l}$ commute with $X_{0,0}$. Indeed,
\begin{eqnarray*}	[X_{0,0},a_{2,2}^{k+l}X_{k,l}]&=&X_{0,0}(a_{2,2}^{k+l})X_{k,l}+a_{2,2}^{k+l}X_{0,0}X_{k,l}-a_{2,2}^{k+l}X_{k,l}X_{0,0}\\
	&=&(k+l)a_{2,2}^{k+l}X_{k,l}+a_{2,2}^{k+l}[X_{0,0},X_{k,l}]=0.
\end{eqnarray*}
Therefore these vector fields induce a family of vector fields $Y_{k,l}$, $k+l\geq 1$ on the quotient $M^{(n)}$ which generate the distribution $\mathcal C'$ on $M^{(n)}$. The latter is still integrable by $\tau'=\tau-(p-5)$ rational first integrals: indeed the first integrals $f_{k,l}$ are still non constant first integrals for $\mathcal C'$ excepted the $(p-5)$ first integrals $f_{2,l}$ which are now constant on $M^{(n)}$. They define a complete system of invariants for plane curves with $p$ smooth transverse branches.

%\subsection{Description of $\mathcal{C}$}
%
%	\begin{theorem} The distribution $\mathcal C$ is rationally integrable: there exists $\tau$ independent rational firts integrals, where $\tau$ is the codimension of the generic leaves. This codimension is given by
%\begin{eqnarray*}
%	\tau&=&\frac{(p-2)^2}{4} \mbox{ if $p$ is even,}\\
%	&=& \frac{(p-1)(p-3)}{4} \mbox{ if $p$ is odd.}
%\end{eqnarray*}
%\end{theorem}
%
%
%\subsection{Algorithm and examples}
%
%The algorithm which computes the distribution $\mathcal{C}$ is based upon the lemma(). It appears that, although the expression $x^ky^lN^{(1)}_a$ and $X_{k,l}\cdot N_a$ are algebraic, such a relation won't be in general satisfied with $Z_{k,l}$ algebraic. It has to have an infinite development except in the case $k=l=0$ where one can verify that 
%$$N^{(1)}_a-\left(x\fraction{\partial}{\partial x}+y\fraction{\partial}{\partial y}\right)\cdot N^{(1)}_a=X_{0,0}\cdot N^{(1)}_a.$$
%Nevertheless, what allows us to compute the vector field $X_{k,l}$ on a machine is the finite determinacy property: once the parameter $a_{1\cdot}$ are fixed, the vector fields $X_{k,l}$ are determinated by a finite jet of $Z_{k,l}$. The remain terms of high orders in a relation of type (\ref{compute}) appears only in the term $Z_{k,l}\cdot N^{(1)}_a$. 

\bigskip
\begin{minipage}{0.49\linewidth}
Y. Genzmer\\
{\scriptsize I.R.M.A.\\
Universit\'e de Strasbourg\\
7 Avenue Ren\'e Descartes\\
67084 Strasbourg Cedex \\
genzmer@math.u-strasbg.fr}
\end{minipage}
\begin{minipage}{0.49\linewidth}
E. Paul\\
{\scriptsize Institut de Math\'ematiques de Toulouse\\
Universit\'{e} Paul Sabatier \\
118 route de Narbonne \\
31062 Toulouse cedex 9, France.\\
paul@math.univ-toulouse.fr}
\end{minipage}

\end{document}